\documentclass[preprint]{imsart}

\usepackage{natbib}
\bibpunct{(}{)}{,}{a}{}{;}

\usepackage[T1]{fontenc}
\usepackage[latin1]{inputenc}
\usepackage[english]{babel}
\usepackage{amsmath,amsthm,amssymb,amsfonts,epic,latexsym,hyperref
}
\usepackage{mathabx}
\usepackage[dvips]{graphicx}
\usepackage[usenames]{color}

\usepackage{lipsum}
\usepackage{rotating}
\usepackage{dsfont}
\usepackage{geometry}
\usepackage{stmaryrd}
\usepackage{ifthen}
\usepackage[nottoc]{tocbibind}
\usepackage{verbatim}
\usepackage{tabularx}
\usepackage{framed}
\usepackage{longtable}
\usepackage{tabu}
\usepackage{mathtools}

\allowdisplaybreaks

\newcommand\un{\mathds{1}}
\newcommand\eps{\varepsilon}
\renewcommand\phi{\varphi}
\newcommand\pro[1]{P\left(#1\right)}
\newcommand\esp[1]{\mathbb{E}\left[#1\right]}
\newcommand\proc[2]{\mathbb{P}\left(\left.#1\right|#2\right)}

\newcommand\espc[2]{\mathbb{E}\left[\left.#1\right|#2\right]}

\newcommand\uno[1]{\un_{\left\{#1\right\}}}

\newcommand\indiq{\un}

\newcounter{comptetape}
\newtheorem{thm}{Theorem}[section]
\newtheorem{prop}[thm]{Proposition}
\newtheorem{cor}[thm]{Corollary}
\newtheorem{assu}{Assumption}
\newtheorem{lem}[thm]{Lemma}
\newtheorem{rem}[thm]{Remark}
\newtheorem{defi}[thm]{Definition}

\newcommand\deriv[4]{\frac{\partial\ifthenelse{\equal{#1}1}{}{^{#1}} \bar{X}_{#2}^{#3,#4}(x)}{\partial x_{k,n_k}\ifthenelse{\equal{#1}1}{}{^{#1}}}}
\newcommand\derivf[2]{\left(\sqrt{f_{#2}}\right)^{\ifthenelse{\equal{#1}{1}}{'}{(#1)}}}

\newcommand\diff[3]{\frac{\partial\ifthenelse{\equal{#2}1}{}{^{#2}}#1}{\partial#3\ifthenelse{\equal{#2}1}{}{^{#2}}}}

\newcommand\n{\mathbb{N}}
\renewcommand\r{\mathbb{R}}
\newcommand\R{\r}
\newcommand\E{\mathbb{E}}

\newcommand\cB{\mathcal B}
\newcommand\cP{\mathcal P}

\renewcommand\ll{\left}
\newcommand\rr{\right}
\newcommand\D{D}
\newcommand\N{\n}
\newcommand\bN{\pi}
\newcommand\W{\mathcal{W}}
\newcommand\auxY{\widetilde{X}}
\newcommand\limY{\bar{X}}
\newcommand\mesuremu{\nu}
\newcommand\limYu{\widehat{X}}
\newcommand\limYd{\widecheck{X}}
\newcommand\limmuu{\widehat{\mu}}
\newcommand\limmud{\widecheck{\mu}}

\begin{document}

\begin{frontmatter}

\title{Conditional propagation of chaos for mean field systems of interacting neurons}

\runtitle{Conditional propagation of chaos}

\begin{aug}

\author{\fnms{Xavier} \snm{Erny}\thanksref{m4}\ead[label=e4]{xavier.erny@univ-evry.fr}},
\author{\fnms{Eva}
\snm{L\"ocherbach}\thanksref{m2}\ead[label=e5]{locherbach70@gmail.com}}
\and \author{\fnms{Dasha} \snm{Loukianova}\thanksref{m4} \ead[label=e4]{dasha.loukianova@univ-evry.fr}}

\address{
  \thanksmark{m4}Universit\'e Paris-Saclay, CNRS, Univ Evry, Laboratoire de Math\'ematiques et Mod\'elisation d'Evry, 91037, Evry, France\\
  \thanksmark{m2}Statistique, Analyse et Mod\'elisation Multidisciplinaire, Universit\'e Paris 1 Panth\'eon-Sorbonne, EA 4543 et FR FP2M 2036 CNRS}

\runauthor{X. Erny et al.}

 \end{aug}

\begin{abstract}
We study the stochastic system of interacting neurons introduced in
 \cite{de_masi_hydrodynamic_2015} and in \cite{fournier_toy_2016} in a diffusive scaling.
The system consists of $N$ neurons, each spiking randomly with rate depending on its membrane potential. 
At its spiking time, the potential of the spiking neuron  is reset to $0$ and all other neurons
receive an additional amount of potential which is a centred random variable of order $ 1 / \sqrt{N}.$ In between successive spikes, each neuron's potential follows a 
deterministic flow. We prove the convergence of the system, as $N \to \infty$, to a limit nonlinear
jumping stochastic differential 
equation driven by Poisson random measure and an additional Brownian motion $W$  which is created by the central limit theorem. This Brownian motion is underlying each particle's motion and induces a common noise factor for all neurons in the limit system. 
Conditionally on $W,$ the different neurons are independent in the limit system. This is the  {\it conditional propagation of chaos} property. We prove the well-posedness of the limit equation by adapting the ideas of \cite{graham_mckean-vlasov_1992-1} to our frame. To prove the convergence in distribution of the finite system to the limit system, we introduce a new martingale problem that is well suited for our framework. The uniqueness of the limit is deduced from the exchangeability of the underlying system.
\end{abstract}

\begin{keyword}[class=MSC]
	\kwd{60J75}
  \kwd{60K35}
	\kwd{60G55}
	 \kwd{60G09}
  \end{keyword}

\begin{keyword}
 \kwd{Multivariate nonlinear Hawkes processes with variable length memory}
 \kwd{Mean field interaction}
 \kwd{Piecewise deterministic Markov processes}
 \kwd{Interacting particle systems}
 \kwd{Propagation of chaos}
 \kwd{Exchangeability}
 \kwd{Hewitt Savage theorem}
\end{keyword}

\end{frontmatter}

\section*{Introduction}

This paper is devoted to the study of  the Markov process $X^N_t = (X^{N, 1 }_t, \ldots , X^{N, N}_t )$ taking values in $\r^N$  and having generator $A^N$ which is defined for any smooth test function $ \varphi : \R^N \to \R $ by 
$$
A^N  \varphi ( x) = - \alpha \sum_{i=1}^N \partial_{x^i} \varphi (x) x^i + \sum_{i=1}^N f (x^i) \int_\R \nu ( du ) \left( \varphi ( x - x^i e_i + \sum_{j\neq i } \frac{u}{\sqrt{N}} e_j ) - \varphi ( x) \right) ,
$$
where $ x= (x^1, \ldots, x^N) $ and where $ e_j $ denotes the $j-$th unit vector in $ \R^N.$ In the above formula, $ \alpha > 0 $ is a fixed parameter and $ \nu $ is a centred probability measure on $\R$ having a second moment.

Informally, the process $(X^{N,j})_{1\leq j\leq N}$ solves
\begin{equation}
\label{eq:dynintro}
X^{N, i}_t = X^{N,i}_0 - \alpha \int_0^t X^{N, i}_s   ds -  \int_0^t X^{N, i}_{s-} dZ^{N,i}_s+\frac{1}{\sqrt{N}}\sum_{ j \neq i } \int_0^t U^j(s)dZ^{N,j}_s,
\end{equation} 
where $U^j(s)$ are i.i.d. centred random variables distributed according to $ \nu , $ and where for each $1\leq j\leq N,$ $Z^{N,j}$ is a simple point process on $\r_+$ having stochastic intensity $s\mapsto f\ll(X^{N,j}_{s-}\rr).$

The particle system \eqref{eq:dynintro} is a version of the model of interacting neurons considered in \cite{de_masi_hydrodynamic_2015}, inspired by  \cite{galves_infinite_2013}, and then further studied in  \cite{fournier_toy_2016} and \cite{cormier_long_2019}. 
The system consists of $N$ interacting neurons. In \eqref{eq:dynintro}, $Z^{N,j}_t$ represents the number of spikes emitted by the neuron~$j$ in the interval~$[0,t]$ and $X^{N,j}_t$ the membrane potential of the neuron~$j$ at time~$t$. Spiking occurs randomly following a point process of rate $f (x) $ for any neuron of which the membrane potential equals $x.$  Each time a neuron emits a spike, the potentials of all other neurons receive an additional amount of potential. In \cite{de_masi_hydrodynamic_2015}, \cite{fournier_toy_2016} and \cite{cormier_long_2019} this amount is of order $N^{-1}, $ leading to classical mean field limits as $ N \to \infty .$ On the contrary to this, in the present article we study a {\it diffusive scaling} where each neuron $j$ receives the amount $U/\sqrt{N}$ at spike times $t$ of neuron $i, i \neq j,$ where $U \sim \nu $ is a random variable. The variable $U$ is centred modeling the fact that the synaptic weights are balanced.  Moreover, right after its spike, the potential of the spiking neuron~$i$ is reset to~0, interpreted as resting potential. Finally, in between successive spikes, each neuron has a loss of potential of rate~$\alpha$.

Equations similar to \eqref{eq:dynintro} appear also in the frame of multivariate Hawkes processes with mean field interactions. Indeed, if $\ll(Z^{N,i}\rr)_{1\leq i\leq N}$ is a multivariate Hawkes process where the stochastic intensity of each $Z^{N,i}$ is given by $f\ll(X^N_{t-}\rr)_t$ with
\begin{equation}\label{eq:Xold}
X^{N}_t= e^{- \alpha t } X^N_0 + \frac{1}{\sqrt{N}}\sum_{j=1}^N\int_{0}^te^{-\alpha(t-s)}U^j(s)dZ^{N,j}_s,
\end{equation}
then $X^N$ satisfies
$$X_t^N=X_0^N-\alpha \int_0^tX_s^Nds+\frac{1}{\sqrt{N}}\sum_{j=1}^N\int_0^tU^j(s)dZ^{N,j}_s,$$
which corresponds to equation \eqref{eq:dynintro}  without the big jumps, i.e.\ without the reset to $0$ after each spike.

The above system of interacting Hawkes processes with intensity given by \eqref{eq:Xold} has been studied in our previous paper \cite{erny_mean_2019}. There we have  shown firstly that $X^N$ converges in distribution in $D(\r_+,\r)$ to a limit process $\bar{X}$ solving
\begin{equation}
\label{barX}
d\bar{X}_t=-\alpha\bar{X}_tdt+\sigma\sqrt{f\ll(\bar{X}_t\rr)}dW_t ,
\end{equation}
and secondly that the sequence of multivariate counting processes $\ll(Z^{N,i}\rr)_i$ converges in distribution in $D(\r_+,\r)^{\n^*}$ to a limit sequence of counting processes $\ll(\bar{Z}^i\rr)_i .$ Every $\bar{Z}^{i} $ is driven by its own Poisson random measure and has the same intensity $\ll(f(\bar{X}_{t-}) \rr)_t ,$ where $\bar{X}$ is the strong solution of~\eqref{barX} with respect to some Brownian motion~$W$. Consequently, the processes $\bar{Z}^i$ $(i\geq 1)$ are conditionally independent given the Brownian motion~$W.$ 

In the present paper we add the reset term in~\eqref{eq:dynintro} that forces the potential  $X^{N,i}$ of neuron $i$ to go back to~0 at each jump time of~$Z^{N,i}$.  This models the well-known biological fact  that right after its spike,  the membrane potential 
of the spiking neuron is reset to a resting potential. From a mathematical point of view, this reset to $0$ induces a de-synchronization of the processes $X^{N,i}$ ($1\leq i\leq N$). In terms of Hawkes processes, it means that in \eqref{eq:Xold}, the process $ X^N_t$  has been replaced by 
$$X^{N, i }_t = \frac{1}{\sqrt{N}}\sum_{j=1}^N\int_{L^i_t }^te^{-\alpha(t-s)}U^j(s)dZ^{N,j}_s  + e^{ - \alpha t } X_0^{N, i } \indiq_{ L_t^i = 0} , $$
where $ L^i_t = \sup \{ 0\leq s \le t : \Delta Z^{N, i } _s = 1 \} $
is the last spiking time of neuron $i$ before time $t,$ with the convention $ \sup \emptyset := 0.$  Thus the integral over the past, starting from $0 $ in \eqref{eq:Xold}, is replaced by an integral starting at the last jump time before the present time.  Such processes are termed being of {\it variable  length memory}, in reminiscence of \cite{rissanen_universal_1983}. They are the continuous-time analogues of the model considered in \cite{galves_infinite_2013}, and we are thus considering {\it multivariate Hawkes processes with mean field interactions and variable length memory}.  As a consequence, on the contrary to the situation in \cite{erny_mean_2019}, the point processes $Z^{N,i}$ ($1\leq i\leq N$) do not share the same stochastic intensity. The reset term in \eqref{eq:dynintro} is a jump term that survives in the limit $ N \to \infty .$

Before introducing the exact limit equation for the system~\eqref{eq:dynintro}, let us explain informally how the limit particle system associated to $\ll(X^{N,i}\rr)_{1\leq i\leq N}$ should a priori look like.
Suppose for the moment that we already know that there exists a process $ (\limY^1, \limY^2 , \limY^3, \ldots ) \in \D ( \R_+, \R)^{\n^*} $ such that for all $ K > 0, $ weak convergence $ {\mathcal L }(X^{N, 1, }, \ldots , X^{N, K} ) \to {\mathcal L} ( \limY^1, \ldots, \limY^K) $ in $\D (\R_+, \R )^K ,$ as $ N \to \infty ,$ holds.  In equation~\eqref{eq:dynintro} the only term that depends on~$N$ is the martingale term which is approximately given by
$$M_t^N=\frac{1}{\sqrt{N}}\sum_{j=1}^N\int_0^tU^j(s)dZ^{N,j}_s.$$
Then in the infinite neuron model, each process $\bar{X}^i$ should solve the equation~\eqref{eq:dynintro}, where the term $M_t^N$ is replaced by $M_t:=\underset{N\to\infty}{\lim}M_t^N$. Because of the scaling in $N^{-1/2},$ the limit martingale $M_t$ will be a stochastic integral with respect to some Brownian motion, and its variance  the limit of 
$$\esp{  (M^N_t )^2} =  \sigma^2  \int_0^t \esp{\frac1N \sum_{j=1}^N  f ( X_s^{N, j } ) } ds ,$$ 
where $\sigma^2$ is the variance of $U^j(s).$ Therefore, the limit martingale (if it exists) must be of the form
$$M_t=\sigma\int_0^t\sqrt{\underset{N\to\infty}{\lim}\frac{1}{N}\sum_{j=1}^Nf\ll(X_s^{N,j}\rr)} \, dW_s=\sigma\int_0^t\sqrt{\underset{N\to\infty}{\lim}\mu^N_s(f)}dW_s,$$
where $\mu_s^N$ is the empirical measure of the system $\ll(X^{N,j}_s\rr)_{1\leq j\leq N}$ and $W$ is a one-dimensional standard Brownian motion.

Since the law of  the $N-$particle system $ (X^{N, 1}, \ldots, X^{N, N} ) $ is symmetric, the law of the limit system $ \limY  =   (\limY^1, \limY^2 , \limY^3, \ldots ) $ must be exchangeable, that is, for all finite permutations $\sigma, $ we have that ${\mathcal L} ( \limY^{\sigma ( 1) }, \limY^{\sigma ( 2) } , \ldots ) = {\mathcal L} (\limY).$ In particular, the theorem of Hewitt-Savage, see \cite{hewitt_symmetric_1955}, implies that the random limit 
\begin{equation}\label{eq:limitmu}
 \mu_s  := \lim_{N\to \infty}\frac1N \sum_{i=1}^N \delta_{\limY^i_s } 
\end{equation} 
exists. Supposing that $\mu^N_s$ converges, it necessarily converges towards $\mu_s.$ Therefore,  $\limY $ should solve the limit system 
\begin{equation}
\label{eq:dynlimintro}
\limY^i_t = \limY^i_0 -\alpha   \int_0^t \limY^i_s   ds - \int_0^t \limY^i_{s- } d\bar{Z}^{i}_s + \sigma \int_0^t \sqrt{ \mu_s ( f) } d W_s , i \in \N,
\end{equation}
where each $ \bar{Z}^i $ has intensity $ t \mapsto f ( \limY^i_{t- }),$  and where $\mu_s$ is given by \eqref{eq:limitmu}.

The above arguments are made rigorous in Sections \ref{sec:21} and \ref{sec:22} below. 

Let us briefly discuss the form of the limit equation \eqref{eq:dynlimintro}. 
Analogously to \cite{erny_mean_2019}, the scaling in $N^{-1/2}$ in~\eqref{eq:dynintro} creates a Brownian motion~$W$ in the limit system~\eqref{eq:dynlimintro}. We will show that the presence of this Brownian motion entails a {\it conditional propagation of chaos}, that is the conditional independence of the particles given~$W$. In particular, the limit measure $\mu_s$ will be random. This differs from the classical framework, where the scaling is in $N^{-1}$ (see e.g. \cite{delattre_hawkes_2016}, \cite{ditlevsen_multi-class_2017} in the framework of Hawkes processes, and \cite{de_masi_hydrodynamic_2015}, \cite{fournier_toy_2016} and \cite{cormier_long_2019} in the framework of systems of interacting neurons), leading  to a deterministic limit measure $ \mu_s $ and the true propagation of chaos property implying that  the particles of the limit system are independent. 

This is not the first time that conditional propagation of chaos is studied in the literature; it has already been considered e.g. in \cite{carmona_mean_2016}, \cite{coghi_propagation_2016} and \cite{dermoune_propagation_2003}. But in these papers the common noise, represented by a common (maybe infinite dimensional) Brownian motion, is already present at the level of the finite particle system, the mean field interactions act on the drift of each particle, and the scaling is the classical one in~$N^{-1}.$ On the contrary to this, in our model, this common Brownian motion, leading to conditional propagation of chaos, is only present in the limit, and it is created by the central limit theorem as a consequence of the joint action of the small jumps of the finite size particle system. Moreover, in our model, the interactions survive as a variance term in the limit system due to the diffusive scaling in $N^{-1/2}.$

Now let us discuss the form of $\mu_s$, which is the limit of the empirical measures of the limit system~$\ll(\limY^i_s\rr)_{i\geq 1}$. The theorem of  Hewitt-Savage, \cite{hewitt_symmetric_1955},  implies that the law of $\ll(\limY_s^i\rr)_{i\geq 1}$ is a mixture directed by the law of $\mu_s$. As it has been remarked by \cite{carmona_mean_2016} and \cite{coghi_propagation_2016}, this conditioning reflects the dependencies between the particles.

We will  show that the  variables $\limY^i$ are conditionally independent given the Brownian motion $W.$ As a consequence, $\mu_s $ will be shown to be the conditional law of the solution given the Brownian motion, that is, $P-$almost surely,
\begin{equation}\label{eq:limitmudef}
  \mu_s ( \cdot ) = P  ( \limY^i_s \in \cdot | (W_t)_{ 0 \le t \le s  }  )  = P( \limY^i_s \in \cdot | W ) ,
\end{equation}  
for any $ i \in \N .$ Equation \eqref{eq:dynlimintro} together with \eqref{eq:limitmudef} gives a precise definition of the limit system. 

The nonlinear SDE \eqref{eq:dynlimintro} is not clearly well-posed, and our first main result, Theorem \ref{prop:42}, gives appropriate conditions on the system that guarantee pathwise uniqueness and the existence of a strong solution to \eqref{eq:dynlimintro}.

We then prove, in Sections \ref{sec:21} and \ref{sec:22}, our main Theorem \ref{convergencemuN} stating the convergence in distribution of the sequence of empirical measures $ \mu^N=N^{-1}\sum_{i=1}^N \delta_{(X^{N,i}_t)_{t\geq 0}}, $ in ${\mathcal P} (D(\r_+,\r)),$ to the random limit $ \mu = P ( (\limY_t)_{t\geq 0} \in \cdot | W) .$ To do so, we first prove that under suitable conditions on the parameters of the system, the sequence $ \mu^N$ is tight (see Proposition \ref{mutight} below). We then follow a classical road and identify every possible limit as solution of a martingale problem. Since the random limit measure $ \mu $ will only be the directing measure of the limit system (that is, the conditional law of each coordinate, but not its law), this martingale problem is not a classical one. It is in particular designed to reflect the correlation between the particles and to describe all possible limits of couples of neurons.
 
Classical representation theorems imply that any coordinate of the limit process must satisfy an equation of the type \eqref{eq:dynlimintro}. The fact that our martingale problem describes correlations within couples of neurons allows to show that each coordinate is driven by its own Poisson random measure and that all coordinates are driven by the same underlying Brownian motion $W.$  But it is not yet clear that  $\mu_s $ is of the form \eqref{eq:limitmudef}. In other words, it has to be proven that the only common randomness is the one present in the driving Brownian motion $W.$ To prove this last point, we introduce an auxiliary particle system which is a mean field particle version of the limit system, constructed with the same underlying Brownian motion, and we provide an explicit control on the distance (with respect to a particular $L^1 -$norm) between the two systems.

Let us finally mention that the random limit measure $ \mu $ satisfies the following nonlinear stochastic PDE in weak form: for any test function
$ \phi \in C^2_b (\R),$ the set of $C^2$-functions on $\R$ such that $\phi,$ $\phi'$  and $\phi''$ are bounded,
for any $t\geq 0$,
\begin{multline*}
\int_\R  \phi (x) \mu_t (dx)  = \int_\R  \phi (x) \mu_0  (dx) +  \int_0^t \left( \int_\R  \phi' (x) \mu_s (dx) \right) \,  \sqrt{\mu_s (f) } d W_s \\
+\int_0^t \int_\R  
\Big([ \phi ( 0) - \phi ( x) ] f(x) - \alpha  \phi'(x) x + \frac12 \phi'' (x) \mu_s (f)  \Big) \mu_s ( dx)ds .
\end{multline*}

{\bf Organisation of the paper.} In Section~\ref{secnota}, we state the assumptions and formulate the main results. Section~\ref{secthm14} is devoted to the proof of the convergence of $\mu^N:=\sum_{j=1}^N\delta_{X^{N,j}}$ (Theorem~\ref{convergencemuN}). In particular, we introduce our new martingale problem in Section \ref{sec:22} and prove the uniqueness of the limit law in Theorem \ref{uniquelimit}. Finally, in Appendix, we prove some auxiliary results.

\section{Notation, Model and main results}
\label{secnota}

\subsection{Notation}
We use the following notation throughout the paper.

If $E$ is a metric space, we note:
\begin{itemize}
\item $\mathcal{P}(E)$ the space of probability measures on $E$ endowed with the topology of the weak convergence,
\item $C_b^n(E)$ the set of the functions $g$ which are $n$ times continuously differentiable such that $g^{(k)}$ is bounded for each $0\leq k\leq n,$
\item $C_c^n(E)$ the set of functions $g\in C_b^n(E)$ that have a compact support.
\end{itemize}

In addition, in what follows $D(\r_+,\r)$ denotes the space of c\`adl\`ag functions from $\r_+$ to $\r$, endowed with Skorohod metric, and $C$ and $K$ denote arbitrary positive constants whose values can change from line to line in an equation. We write $C_\theta$ and $K_\theta$ if the constants depend on some parameter $\theta.$

\subsection{The finite system}

We consider, for each $N\geq 1$, a family of i.i.d. Poisson measures $(\pi^i(ds, dz, du))_{i=1,\dots,N}$ on $\r_+  \times \r_+ \times \r $ having intensity measure $ds dz \mesuremu (du)$ where $\nu$ is a probability measure on $\r$, as well as an i.i.d. family $(X^{N,i}_0)_{i=1,\dots,N}$ of $\r $-valued random variables independent of the Poisson measures. The object of this paper is to study the convergence of the Markov process $X^N_t = (X^{N, 1 }_t, \ldots , X^{N, N}_t )$ taking values in $\r^N$ and solving, for $i=1,\dots,N$, for $t\geq 0$,
\begin{equation}\label{eq:dyn}
\ll\{\begin{array}{rcl}
X^{N, i}_t &= &\displaystyle X^{N,i}_0  - \alpha   \int_0^t  X^{N, i}_s   ds -  \int_{[0,t]\times\r_+\times\r} X^{N, i}_{s-}  \indiq_{ \{ z \le  f ( X^{N, i}_{s-}) \}} \pi^i (ds,dz,du) \\
&&+\displaystyle \frac{1}{\sqrt{N}}\sum_{ j \neq i } \int_{[0,t]\times\r_+\times\r}u \indiq_{ \{ z \le  f ( X^{N, j}_{s-}) \}} \pi^j (ds,dz,du),\\
X_0^{N,i}&\sim& \nu_0.
\end{array}\right.
\end{equation} 
The coefficients of this system are the exponential loss factor $ \alpha > 0, $  the jump rate function $f:\r \mapsto \r_+$  and the probability measures $\mesuremu$ and $\nu_0$.

In order to guarantee existence and uniqueness of a strong solution of~\eqref{eq:dyn}, we introduce the following hypothesis.

\begin{assu}
\label{ass:1}
The function $f$ is Lipschitz continuous. 
\end{assu}

In addition, we also need the following condition to obtain a priori bounds on some moments of the process $\ll(X^{N,i}\rr)_{1\leq i\leq N}.$ 
\begin{assu}
\label{control}
We assume that 
 $\int_\r xd\nu(x)=0,$ $\int_\r x^2d\nu(x)<+\infty,$ and $\int_\r x^2d\nu_0(x)<+\infty.$
\end{assu}

Under Assumptions~\ref{ass:1} and~\ref{control}, existence and uniqueness of strong solutions of~\eqref{eq:dyn} follow from Theorem~IV.9.1 of \cite{ikeda_stochastic_1989}, exactly in the same way as in Proposition~6.6 of \cite{erny_mean_2019}.

We now define precisely the limit system and discuss its properties before proving the convergence of the finite to the limit system. 

\subsection{The limit system}
The limit system $\ll(\limY^i\rr)_{i\geq 1}$ is given by
\begin{equation}\label{eq:dynlimit1}
\ll\{\begin{array}{rcl}
\limY^i_t &=& \displaystyle\limY^i_0 - \alpha   \int_0^t \limY^i_s   ds -\int_{[0,t]\times\r_+\times\r}  \limY^i_{s- } \indiq_{ \{ z \le  f ( \limY^i_{s-}) \}} \bN^i  (ds,dz, du) \\
&&\displaystyle + \sigma \int_0^t \sqrt{\espc{f\ll(\limY^i_s\rr)}{\W  }} d W_s,\\
\limY^i_0&\sim&\nu_0 .
\end{array}\rr. 
\end{equation}
In the above equation,  $(W_t)_{t\geq 0}$ is a standard one-dimensional Brownian motion which is independent of the Poisson random measures, and $ \W = \sigma \{ W_t ,t \geq 0 \} .$ Moreover,  the initial positions $ \limY^i_0  , i \geq 1 , $ are i.i.d., independent of $ W$ and of the Poisson random measures, distributed according to $\nu_0$ which is the same probability measure as in~\eqref{eq:dyn}. The common jumps of the particles in the finite system, due to their scaling in $ 1/\sqrt{N} $ and the fact that they are centred, by the Central Limit Theorem, 
create this single Brownian motion $ W_t $ which is underlying each particle's motion and which induces the common noise factor for all particles in the limit. 

The limit equation \eqref{eq:dynlimit1} is not clearly well-posed and requires more conditions on the rate function~$f$. Let us briefly comment on the type of difficulties that one encounters when proving trajectorial uniqueness of \eqref{eq:dynlimit1}. Roughly speaking, the jump terms demand to work in an $L^1 - $framework, while the diffusive terms demand to work in an $L^2-$framework. \cite{graham_mckean-vlasov_1992-1} proposes a unified approach to deal both with jump and with diffusion terms in a non-linear framework, and we shall rely on his ideas in the sequel. The presence of the random volatility term which involves conditional expectation causes however additional technical difficulties. 
Finally, another difficulty comes from the fact that the jumps induce non-Lipschitz terms of the form $ \limY^i_s   f ( \limY^i_{s}) .$  For this reason a classical Wasserstein-$1-$coupling is not appropriate for the jump terms. Therefore we propose a different distance which is inspired by the one already used in \cite{fournier_toy_2016}. To do so,  we need to work under the following additional assumption.

\begin{assu}\label{ass:2}
1. We suppose that $ \inf f > 0 .$ \\
2. There exists a function $a \in C^2(\R , \R_+ ), $ strictly increasing and bounded, such that,  for a suitable constant $C,$ for all $ x, y \in \R,$ 
$$ |a'' ( x) - a'' (y) | +  |a'(x) - a' (y) | + |x a'(x) - ya'(y) | + |f(x)-f(y)| \le C | a(x)-a(y) |.$$
\end{assu}

Note that Assumption~\ref{ass:2} implies Assumption~\ref{ass:1} as well as the boundedness of the rate function~$f.$ 

\begin{prop}
Suppose that  $ f(x) =  c + d  \arctan (x) ,$ where $c > d \frac{\pi}2 ,$ $ d > 0 .$ Then Assumption \ref{ass:2} holds with $ a =  f.$
\end{prop}

\begin{proof}
We quickly check that $|x a'(x) - ya'(y) |  \le C | a(x) - a(y) | .$ We have that $ a' (x) =  \frac{d}{1+x^2 } , $ whence $ x a'(x) - ya'(y) = d( \frac{x}{1+ x^2 } - \frac{y}{1+ y^2 }) .$ We use that $ \left| \frac{d}{dx} \left( \frac{x}{1+ x^2 } \right)\right|  =\left| \frac{1- x^2}{(1+x^2)^2} \right| \le \frac{1}{1+x^2 } .$ Suppose w.l.o.g. that $ x \le y .$ As a consequence, 
$$ | x a'(x) - ya'(y)| = d \left| \int_x^y \frac{1- t^2}{(1+t^2)^2}dt \right| \le d \int_x^y \frac{1}{1 + t^2 } dt =  | \arctan (y) - \arctan (x)| = d | a(x) - a(y)|. $$ 
The other points of Assumption \ref{ass:2} follow immediately.
\end{proof}

Under these additional assumptions we obtain the well-posedness of each coordinate of the limit system~\eqref{eq:dynlimit1}, 
that is, of the $({\mathcal F}_t)_t- $ adapted process $ (\limY_t)_t  $ having c\`adl\`ag trajectories which is solution of the SDE
\begin{equation}
\ll\{\begin{array}{rcl}
\label{eq:dynlimit}
d\limY_t&=&- \alpha \displaystyle \limY_t dt- \limY_{t-}\int_{ \r_+\times\r}  \indiq_{ \{ z \le  f ( \limY_{t-}) \}} \bN  (dt,dz, du) +\sigma\sqrt{\mu_t(f)}dW_t,\\
\limY_0&\sim&\nu_0 ,~~\mu_t(f)=\espc{f\ll(\limY_t\rr)}{\W }=\espc{f\ll(\limY_t\rr)}{\W_t } .
\end{array}\rr.
\end{equation}
Here, $ {\mathcal F}_t = \sigma \{  \bN ( [0, s ] \times A ), s \le t , A \in {\mathcal B} ( \R_+ \times \R ) \} \vee  \W_t, $ $\W_t = \sigma \{ W_s , s \le t \} $ and $ \W = \sigma \{ W_s , s \geq 0\}.$ 

\begin{thm}\label{prop:42}
Grant Assumption~\ref{ass:2}.  \\
1. Pathwise uniqueness holds for the nonlinear SDE~\eqref{eq:dynlimit}. \\
2. If additionally, $\int_\r x^2d\nu_0(x)<+\infty,$ then there exists a unique strong solution $(\limY_t)_{t\geq 0}  $ of the nonlinear SDE~\eqref{eq:dynlimit}, which is $({\mathcal F}_t)_t- $ adapted with c\`adl\`ag trajectories, satisfying for every $t>0$,
\begin{equation}
\label{controllimite}
\esp{\underset{0\leq s\leq t}{\sup}\limY_s^2}<+\infty.
\end{equation}
\end{thm}

\begin{rem}
Notice that the stochastic integral $ \int_0^t \sqrt{\mu_s(f)}dW_s $ is well-defined since $ s \mapsto \sqrt{\mu_s(f)} $ is an $({\mathcal W}_t)_t- $progressively measurable process.
\end{rem}

In what follows we just give the proof of Item 1. of the above theorem since its arguments are important for the sequel. We postpone the rather classical proof of Item 2.\  to Appendix.

\begin{proof}[Proof of Item 1. of Theorem \ref{prop:42}]
Consider two solutions $ (\limYu_t)_{t \geq 0}$ and $ (\limYd_t)_{t \geq 0 } , $  $({\mathcal F}_t)_t- $adapted, defined on the same probability space and driven by the same Poisson random measure $\bN $ and the same Brownian motion~$ W,$ and with $ \limYu_0 = \limYd_0.$  We consider $ Z_t := a(\limYu_t) -   a( \limYd_t) .$ Denote $\limmuu_s(f)=\E[f( \limYu_s)|\W_s]$ and $\limmud_s(f)=\E[f( \limYd_s)|\W_s]. $  \\
Using Ito's formula, we can write
\begin{multline*}
Z_t = - \alpha \int_0^t \left(  \limYu_s   a' ( \limYu_s ) - \limYd_s a' ( \limYd_s) \right) ds +\frac12 \int_0^t ( a'' ( \limYu_s) \limmuu_s(f)  - a'' ( \limYd_s )  \limmud_s(f) ) \sigma^2 ds \\
+  \int_0^t  ( a'  ( \limYu_s) \sqrt{\limmuu_s(f)} -a'  (\limYd_s )  \sqrt{ \limmud_s(f)} ) \sigma d W_s  \\
- \int_{[0,t]\times\r_+\times\r} \, [ a (\limYu_{s- }) - a(  \limYd_{s-})  ]    \indiq_{ \{ z \le  f ( \limYu_{s-})  \wedge f ( \limYd_{s-}) \}} \bN   (ds, dz, du)\\
+  \int_{[0,t]\times\r_+\times\r}  \, [a(0)-  a( \limYu_{s-} )]     \indiq_{ \{ f ( \limYd_{s-} ) < z \le f ( \limYu_{s-} )  \}} \bN (ds, dz, du) \\
+  \int_{[0,t]\times\r_+\times\r} \, [  a( \limYd_{s-} ) - a(0) ]  \indiq_{ \{ f (  \limYu_{s-} ) < z \le f ( \limYd_{s-} )  \}} \bN (ds,dz,du)  = : A_t + M_t +\Delta_t ,
\end{multline*}
where $ A_t $ denotes the bounded variation part of the evolution, $M_t$ the martingale part and $ \Delta_t$ the sum of the three jump terms. Notice that 
$$M_t=  \int_0^t  ( a'  ( \limYu_s) \sqrt{\limmuu_s(f)} -a'(\limYd_s )  \sqrt{\limmud_s(f)} ) \sigma d W_s$$
is a square integrable martingale since $ f$ and $a'$ are bounded. 

We wish to obtain a control on $ |Z^* _t | := \sup_{ s\le t } |Z_s | .$ We first take care of the jumps of $ |Z_t|.$ Notice first that, since $f$ and $a$ are bounded, 
\begin{multline*}
\Delta  (x,y):=  (f(x) \wedge f(y)) | a (x) - a(y ) |  + | f (x  ) - f( y )  |  \; \Big|  |a (0) - a(y)| + |a(0) - a(x)| \Big| \\
\le C | a (x) - a( y ) | ,
\end{multline*}
implying that 
$$ \E \sup_{s \le t } | \Delta_s |   \le  C  \E \int_0^t  | a(\limYu_s) - a(\limYd_s ) | ds \le C t \, \E |Z_t^* |  . $$ 
Moreover, for a constant $C$ depending on $\sigma^2 ,$ $ \| f \|_\infty , \|a\|_\infty,  \| a'\|_\infty, \| a'' \|_\infty   $ and $ \alpha  , $ 
\begin{multline*} 
\E \sup_{ s \le t }  | A_s |   \le C  \int_0^t \E  |a'( \limYu_s ) \limYu_s - a' ( \limYd_s )\limYd_s | ds  \\
+ C \left[ \int_0^t | a'' ( \limYu_s )  -a '' ( \limYd_s ) |   ds +  \int_0^t | \limmuu_s ( f) - \limmud_s ( f) | ds \right] .
\end{multline*}
We know that  $ |a'( \limYu_s ) \limYu_s - a' ( \limYd_s )\limYd_s |   +  |a'' ( \limYu_s )  - a'' ( \limYd_s ) | \le C  |a ( \limYu_s  )  - a ( \limYd_s ) |= C  | Z_s| .$ Therefore, 
$$ \E \sup_{ s \le t }  | A_s |  \le C \E \left[ \int_0^t | Z_s | ds + \int_0^t | \limmuu_s ( f) - \limmud_s ( f) | ds \right].$$
Moreover, 
$$  |\limmuu_s (f)- \limmud_s (f) |   = \Big| \E \left( f ( \limYu_s ) - f ( \limYd_s ) | \W \right) \Big|  \le  \E \left(  | f ( \limYu_s ) - f ( \limYd_s )| | \W  \right) \leq \E ( |Z_s|  | \W) ,$$
and thus, 
$$ \E \int_0^t  | \limmuu_s ( f) - \limmud_s ( f) | ds \le  \E \int_0^t |Z_s|  ds \le t \E | Z^*_t| .$$ 
Putting all these upper bounds together we conclude that for a constant $C$ not depending on $t,$ 
$$ \E \sup_{s \le t} |A_s| \le C t \E  |Z_t^*|  .$$   
Finally, we treat the martingale part using the Burkholder-Davis-Gundy inequality, and we obtain 
$$
 \E \sup_{s \le t} |M_s| \le C \E \left[   \left( \int_0^t (a' (\limYu_s ) \sqrt{ \limmuu_s ( f) } - a' (\limYd_s )  \sqrt{ \limmud_s ( f) })^2 ds  \right)^{1/2}\right].$$
But
\begin{multline}\label{eq:varquadratique}  
 (a' (\limYu_s ) \sqrt{ \limmuu_s ( f) } - a' (\limYd_s )  \sqrt{ \limmud_s ( f) })^2  \le 
C \left[ ( (a' (\limYu_s ) - a' (\limYd_s ))^2 +  (\sqrt{ \limmuu_s ( f) } -   \sqrt{ \limmud_s ( f) })^2 \right] \\
\le C | Z_t^*|^2 + C (\sqrt{ \limmuu_s ( f) } -   \sqrt{ \limmud_s ( f) })^2  ,
\end{multline} 
where we have used that $ | a' (x) - a' (y) | \le C | a(x) - a(y) | $ and that $f$ and $a'$ are bounded.  

Finally, since $\inf  f> 0, $
$$| \sqrt{ \limmuu_s ( f) } -   \sqrt{ \limmud_s ( f) }|^2 \le C | \limmuu_s ( f)  -   \limmud_s ( f) |^2  \le  C \left( \E ( |Z_s^*|  | \W_s ) \right)^2.$$
We use that  
$ |Z_s^* |  \le | Z_t^*| ,$ implying that $ \E ( |Z_s^*|  | \W ) \le \E ( |Z_t^*|  | \W ).$ Therefore we obtain the upper bound 
$$ | \sqrt{ \limmuu_s ( f) } -   \sqrt{ \limmud_s ( f) }|^2
 \le C  \left( \E ( |Z_t^*|  | \W ) \right)^2 $$
for all $ s \le t ,$  which implies the control of 
$$  \E \sup_{s \le t} |M_s| \le C \sqrt{t} \E | Z_t^* | .$$ 
The above upper bounds imply that, for a constant $C$ not depending on $t $ nor on the initial condition,
$$ \E |Z_t^*| \le C (t + \sqrt{t} )   \E | Z_t^* |  ,$$
and therefore, for $ t_1 $ sufficiently small, $ \E |Z_{t_1}^*|  = 0 . $ We can repeat this argument on intervals $ [ t_1, 2 t_1 ], $ with initial condition $\hat X_{t_1 } ,$ and iterate it up to any finite $T$ because $t_1 $ does only depend on the coefficients of the system but not on the initial condition. This implies the assertion. 
\end{proof}

\begin{rem}
Theorem~\ref{prop:42} states the well-posedness of the SDE~\eqref{eq:dynlimit}. Under the same hypotheses, with almost the same reasoning, one can prove the well-posedness of the system~\eqref{eq:dynlimit1}.
\end{rem}

In the sequel, we shall also use an important property of the limit system~\eqref{eq:dynlimit1}, which is the conditional independence of the processes $\limY^i$ ($i\geq 1$) given the Brownian motion~$W$.

\begin{prop}
\label{independence}
If Assumption~\ref{ass:2} holds and $\int_\r x^2d\nu_0(x)<+\infty,$ then 
\begin{itemize}
\item[(i)] for all $N\in\N^*$ there exists a strong solution $\ll(\limY^i\rr)_{1\leq i\leq N}$ of~\eqref{eq:dynlimit1}, and pathwise uniqueness holds,
\item[(ii)] $\limY^1,\ldots, \limY^N$ are independent conditionally to $W,$
\item[(iii)] for all $t\geq 0$, almost surely, the weak limit of $\frac1N \sum_{i=1}^N \delta_{\limY^i_{|[0,t]} }  $ is given by $ \lim_{N\to \infty}\frac1N \sum_{i=1}^N \delta_{\limY^i_{|[0,t]} }= P ( \limY^i_{|[0,t]} \in \cdot | \W_{t} )  = P( \limY^i_{|[0,t]} \in \cdot | \W ) .$ 
\end{itemize}
\end{prop}

Let us finally mention that the random limit measure $ \mu $ satisfies a nonlinear stochastic PDE in weak form. More precisely, 
\begin{cor}\label{cor:PDE}
Grant Assumption~\ref{ass:2} and suppose that $\int_\r x^2d\nu_0(x)<+\infty . $ Then the measure $ \mu = P ( (\limY_t)_{t\geq 0} \in \cdot | W) $ satisfies the following nonlinear stochastic PDE in weak form: for any 
$ \phi \in C^2_b (\R),$ for any $t\geq 0$,
\begin{multline*}
\int_\R  \phi (x) \mu_t (dx)  = \int_\R  \phi (x) \nu_0  (dx) +  \int_0^t \left( \int_\R  \phi' (x) \mu_s (dx) \right) \,  \sqrt{\mu_s (f) } \sigma d W_s \\
+\int_0^t \int_\R  
\Big([ \phi ( 0) - \phi ( x) ] f(x) - \alpha  \phi'(x) x + \frac12 \sigma^2 \phi'' (x) \mu_s (f)  \Big) \mu_s ( dx)ds .
\end{multline*}
\end{cor} 
The proofs of Proposition~\ref{independence} and of Corollary \ref{cor:PDE} are postponed to Appendix.

\subsection{Convergence to the limit system}

We are now able to state our main result.
\begin{thm}
\label{convergencemuN}
Grant Assumptions~\ref{ass:1},~\ref{control} and~\ref{ass:2}. Then the empirical measure $\mu^N=\frac1N\sum_{i=1}^N\delta_{X^{N,i}}$ of the $N-$particle system $(X^{N,i})_{1\leq i\leq N}$ converges in distribution in $\mathcal{P}(D(\r_+,\r))$ to $\mu:=\mathcal{L}(\bar X^1|\W)$, where $(\bar X^i)_{i\geq 1}$ is solution of~\eqref{eq:dynlimit1}.
\end{thm}

\begin{cor}
Under the assumptions of Theorem~\ref{convergencemuN}, $(X^{N,j})_{1\leq j\leq N}$ converges in distribution to $(\bar X^j)_{j\geq 1}$ in $D(\r_+,\r)^{\n^*}.$
\end{cor}

\begin{proof}
Together with the statement of Theorem \ref{convergencemuN}, the proof is an immediate consequence of Proposition~7.20 of \cite{aldous_exchangeability_1983}.
\end{proof}
  
We will prove Theorem~\ref{convergencemuN} in a two step procedure. Firstly we prove the tightness of the sequence of empirical measures, and then in a second step we identify all possible limits as solutions of a martingale problem.

\section{Proof of Theorem~\ref{convergencemuN}}
\label{secthm14}

This section is dedicated to prove that the sequence $(\mu^N)_N$ of the empirical measures $\mu^N:=\sum_{j=1}^N\delta_{(X^j_t)_{t\geq 0}}$ converges in distribution to $\mu:=\mathcal{L}(\bar X^1|\W)$, where $(\bar X^j)_{j\geq 1}$ is solution of~\eqref{eq:dynlimit1}.

In a first time, we prove that the sequence $(\mu^N)_N$ is tight on $\mathcal{P}(D(\r_+,\r))$. The main step to prove the convergence of $(\mu^N)_N$ is then to show that each converging subsequence converges to the same limit in distribution. For this purpose, we introduce a new martingale problem, and we show that every possible limit of $\mu^N$ is a solution of this martingale problem. Finally, we will show how the uniqueness of the limit law follows from the exchangeability of the system.

\subsection{Tightness of $(\mu^N)_N$}\label{sec:21}

\begin{prop}\label{mutight}
Grant Assumptions \ref{ass:1} and \ref{control}. For each $N\geq 1$, consider the unique solution $(X^N_t)_{t\geq 0}$ to \eqref{eq:dyn} starting from some i.i.d. $\nu_0$-distributed initial conditions $X^{N,i}_0$.
\begin{itemize}
\item[(i)] The sequence of processes $(X^{N,1}_t)_{t\geq 0}$ is tight in $\D(\R_+, \R)$.
\item[(ii)] The sequence of empirical measures $ \mu^N=N^{-1}\sum_{i=1}^N \delta_{(X^{N,i}_t)_{t\geq 0}}$
is tight in ${\mathcal P}(\D(\R_+, \R))$.
\end{itemize}
\end{prop}

\begin{proof}
First, it is well-known that point (ii) follows from point (i) and the exchangeability 
of the system, see  \cite[Proposition 2.2-(ii)]{sznitman_topics_1989}. We thus
only prove (i). To show that the family $((X^{N,1}_t)_{t\geq 0})_{N\geq 1}$ is tight in $\D(\R_+,\R)$, 
we use the criterion of Aldous, see Theorem~4.5 of \cite{jacod_limit_2003}. It is sufficient to prove that
\begin{itemize}
\item[(a)] for all $ T> 0$, all $\varepsilon >0$,
$ \lim_{ \delta \downarrow 0} \limsup_{N \to \infty } \sup_{ (S,S') \in A_{\delta,T}} 
P ( |X_{S'}^{N, 1 } - X_S^{N , 1 } | > \varepsilon ) = 0$,
where $A_{\delta,T}$ is the set of all pairs of stopping times $(S,S')$ such that
$0\leq S \leq S'\leq S+\delta\leq T$ a.s.,
\item[(b)] for all $ T> 0$, $\lim_{ K \uparrow \infty } \sup_N 
P ( \sup_{ t \in [0, T ] } |X_t^{N, 1 }| \geq K ) = 0$.
\end{itemize}

To check (a), consider $(S,S')\in A_{\delta,T}$ and write
\begin{multline*}
X_{S'}^{N, 1 } - X_S^{N , 1 } =  
- \int_S^{S'} \int_\R \int_0^\infty X^{ N, 1 }_{s- }  \indiq_{\{ z \le f ( X_{s- }^{N, 1} ) \}} \bN^1 (ds, du,  dz ) - \alpha  \int_S^{S'} X^{N, 1 }_s  ds 
\\
+ \frac{1}{ \sqrt{N} } \sum_{j=2}^N \int_S^{S'} \int_\R \int_0^\infty  u  \indiq_{\{ z \le f ( X_{s- }^{N, j} ) \}} 
\bN^j (ds, du, dz ) ,
\end{multline*}
implying that 
\begin{multline*}
|X_{S'}^{N, 1 } - X_S^{N , 1 }|  \le  |  \int_S^{S'} \int_\R \int_0^\infty X^{ N, 1 }_{s- }  \indiq_{\{ z \le f ( X_{s- }^{N, 1} ) \}} \bN^1 (ds, du,  dz ) |  \\
+ \delta\alpha\underset{0\leq s\leq T}{\sup}\ll|X_s^{N,1}\rr|   + |  \frac{1}{ \sqrt{N} } \sum_{j=2}^N \int_S^{S'} \int_\R \int_0^\infty  u  \indiq_{\{ z \le f ( X_{s- }^{N, j} ) \}} 
\bN^j (ds, du, dz ) | \\
=: |I_{S, S'}|+ \delta\alpha\underset{0\leq s\leq T}{\sup}\ll|X_s^{N,1}\rr| + |J_{S, S'}|.
\end{multline*}
We first note that $|I_{S,S'}|>0$ implies that 
$\tilde I_{S,S'}:=
\int_S^{S'} \int_\R \int_0^\infty \indiq_{\{ z \le f ( X_{s- }^{N, 1} ) \}} \bN^i (ds, du, dz)\geq 1$, whence
$$
P ( |I_{S, S'}| > 0 )\leq P (\tilde I_{S,S'}\geq 1)\leq \E[\tilde I_{S,S'}]\le 
\E\Big[ \int_S^{S+\delta} f( X_s^{N, 1 } ) ds \Big] \le ||f||_\infty\delta, 
$$
since $ f$ is bounded. 
We proceed similarly to check that
$$
P ( |J_{S, S'}| \geq \varepsilon ) \le \frac{1}{\varepsilon^2} \E[(J_{S,S'})^2 ]\leq    \frac{\sigma^2}{N\varepsilon^2 } \sum_{j=2}^N\E\Big[ \int_S^{S+\delta} f( X_s^{N, j} ) ds\Big]
\le  \frac{\sigma^2}{\varepsilon^2}  \| f \|_\infty \delta.
$$
The term $\sup_{0\leq s\leq T}|X^{N,1}_s|$ can be handled using Lemma~\ref{estimate}.(ii).

Finally (b) is a straightforward consequence of Lemma~\ref{estimate}.(ii) and Markov's inequality. 

\end{proof}

\subsection{Martingale problem}\label{sec:22}

We now introduce a new martingale problem, whose solutions are the limits of any converging subsequence of $ \mu^N =\frac1{N}\sum_{j=1}^{N}\delta_{X^{N,j}}$. In this martingale problem, we are interested in couples of trajectories to be able to put hands on the correlations between the particles. In particular, this will allow us to show that, in the limit system~\eqref{eq:dynlimit1}, the processes $\bar X^i$ ($i\geq 1$) share the same Brownian motion, but are driven by Poisson measures $\pi^i$ ($i\geq 1$) which are independent. The reason why we only need to study the correlation between two particles is the exchangeability of the infinite system. 

Let $Q$ be a distribution on $\mathcal{P}(D(\r_+,\r))$. Define a probability measure $P$ on $\mathcal{P}(D(\r_+,\r))\times D(\r_+,\r)^2$ by
\begin{equation}\label{eq:P}
P(A\times B):=\int_{\mathcal{P}(D(\r_+,\r))}\un_A( m)  m \otimes m (B)Q(d m).
\end{equation}

We write any atomic event $\omega\in\Omega:=\mathcal{P}(D(\r_+,\r))\times D(\r_+,\r)^2$ as $\omega=(\mu,Y),$ with $Y=(Y^1,Y^2).$ Thus, the law of the canonical variable $ \mu$ is $Q$, and that of $(Y_t)_{t\geq 0}$ is
$$P_Y = \int_{\mathcal{P}(D(\r_+,\r))}Q(d  m)  m \otimes m (\cdot).$$
Moreover we have $P-$ almost surely
$$ \mu=\mathcal{L}(Y^1 | \mu) = \mathcal{L}(Y^2 | \mu) \mbox{ and } \mathcal{L} ( Y| \mu) = \mu \otimes \mu .$$
Writing $ \mu_t := \int_{ D(\r_+,\r) } \mu ( d \gamma ) \delta_{ \gamma_t } $ for the projection onto the $t-$th time coordinate, we introduce the filtration
$$\mathcal{G}_t=\sigma(Y_s,s\leq t)\vee\sigma( \mu_s(f),s\leq t) .$$

\begin{defi}
We say that $Q \in {\mathcal P} ( {\mathcal P} ( D( \R_+, \R ) ) ) $ is a solution to the martingale problem $({\mathcal M}) $ if the following holds.
\begin{itemize}
\item[(i)]  $Q-$almost surely, $\mu_0 = \nu_0.$
\item[(ii)] For all $ g \in C^2_b ( \R^2), $ $M_t^g := g(Y_t) - g(Y_0) - \int_0^t L g( \mu_s, Y_s) ds $
is a $( P,  ({\mathcal G}_t)_t)-$martingale, where 
\begin{align*}
L g (  \mu, x) =& -\alpha x^1\partial_{x^1}g(x)-\alpha x^2\partial_{x^2} g(x)+\frac{\sigma^2}{2} \mu(f) \sum_{i,j=1}^2\partial^2_{x^i x^j}g(x)\\
&+f(x^1)(g(0,x^2)-g(x))+f(x^2)(g(x^1,0)-g(x)).
\end{align*}
\end{itemize}
\end{defi}

\begin{rem}
It is not clear if the martingale problem is well-posed, but we are not interested in proving uniqueness for it. However, we will have uniqueness within the class of all possible limits in distribution of $ \mu^N.$ More precisely, we shall prove that, if $\mu$ is a limit in distribution of $ \mu^N$  such that $\mathcal{L}(\mu)$ is solution to~$(\mathcal{M})$, then $ \mu = {\mathcal L} ( \bar X | \W ) , $ with $\bar X $ the strong solution of \eqref{eq:dynlimit}. Equivalently, defining the problem~$(\mathcal{M})$ for all finite-dimensional distributions, and not only for two coordinates, where $Y^i$ ($i\geq 1$) are defined as a mixture directed by $\mu$, would lead to uniqueness.
\end{rem}

Let $(\limY^i)_{i\geq 1}$ be the solution of the limit system~\eqref{eq:dynlimit1} and $ \mu = \mathcal{L}(\limY^1 |\W)$. Then we already know that $ \mathcal{L}( \mu) $ is a solution of $({\mathcal M}) .$ Let us now characterise any possible solution of $(\mathcal{M}).$

\begin{lem}
\label{representation}
Grant Assumption~\ref{ass:2}. Let $Q \in {\mathcal P} ( {\mathcal P} ( D( \R_+, \R ) ) )$ be a solution of $(\mathcal{M}).$ Let $(\mu , Y)$ be the canonical variable defined above, and write $Y=(Y^1,Y^2)$. Then there exists a standard $(\mathcal{G}_t)_t-$Brownian motion $W$ and on an extension $ (\tilde \Omega, (\mathcal{\tilde G}_t)_t, \tilde P) $ of $ ( \Omega, ({\mathcal G}_t)_t, P) $ there exist $(\mathcal{\tilde G}_t)_t-$ Poisson random measures $\pi^1,\pi^2$ on $\r_+\times\r_+$ having Lebesgue intensity such that $W,\pi^1$ and $\pi^2$ are independent and
\begin{align*}
dY^1_t=&-\alpha Y^1_tdt+\sigma\sqrt{\mu_t(f)}dW_t-Y^1_{t-}\int_{\r_+}\uno{z\leq f(Y^1_{t-})}\pi^1(dt,dz),\\
dY^2_t=&-\alpha Y^2_tdt+\sigma\sqrt{\mu_t(f)}dW_t-Y^2_{t-}\int_{\r_+}\uno{z\leq f(Y^2_{t-})}\pi^2(dt,dz).
\end{align*}
\end{lem}

\begin{proof}
Item (ii) of of $(\mathcal{M})$ together with Theorem~II.2.42 of \cite{jacod_limit_2003} imply that $Y$ is a semimartingale with characteristics $(B,C,\nu)$ given by
\begin{align*}
&B^i_t=-\alpha\int_0^tY^i_sds-\int_0^tY_s^if(Y^i_s)ds,~~1\leq i\leq 2,\\
&C_t^{i,j}=\int_0^t  \mu_s(f) ds,~~1\leq i,j\leq 2,\\
&\nu(dt,dy)=dt(f(Y^1_{t-})\delta_{(-Y^1_{t-},0)}(dy)+f(Y^2_{t-})\delta_{(0,-Y^2_{t-})}(dy)).
\end{align*}

Then we can use the canonical representation of $Y$ (see Theorem~II.2.34 of \cite{jacod_limit_2003}) with the truncation function $h(y)=y$ for every $y$: $Y_t-Y_0-B_t=M^c_t+M^d_t,$ where $M^c$ is a continuous local martingale and $M^d$ a purely discontinuous local martingale. By definition of the characteristics, $\langle M^{c,i},M^{c,j}\rangle_t=C^{i,j}_t.$ In particular, $\langle M^{c,i}\rangle_t=\int_0^t \mu_s(f)ds$ ($i=1,2$). Consequently, applying Theorem~II.7.1 of \cite{ikeda_stochastic_1989} to both coordinates, we know that there exist Brownian motions $W^1,W^2$ such that
$$M^{c,i}_t=\int_0^t\sqrt{ \mu_s(f)}dW^i_s,~~i=1,2.$$
We now prove that $W^1=W^2.$ Let $\rho$ be the correlation between $W^1$ and $W^2$. Classical computations give $\langle W^1,W^2\rangle_t=\rho t,$ implying that $\langle M^{c,1},M^{c,2}\rangle_t=\rho\int_0^t \mu_s(f)ds$. In addition $\langle M^{c,1},M^{c,2}\rangle_t=C^{1,2}_t=\int_0^t \mu_s(f)ds$, and this implies that $\rho=1$ and $W^1=W^2,$ since $\int_0^t \mu_s(f) ds>0$ because $f$ is lower-bounded.

We now prove the existence of the independent Poisson measures $\pi^1,\pi^2.$ We know that $M^d=h*(\mu^Y-\nu),$ where $\mu^Y=\sum_s\uno{\Delta Y_s\neq 0}\delta_{(s,Y_s)}$ is the jump measure of $Y$ and $\nu$ is its compensator. We rely on Theorem~II.7.4 of \cite{ikeda_stochastic_1989}. Using the notation therein, we introduce $Z=\r_+,$ $m$ Lebesgue measure on $Z$ and
$$\theta(t,z):=(-Y^1_{t-},0)\uno{z\leq f(Y^1_{t-})}+(0,-Y^2_{t-})\uno{||f||_\infty<z\leq ||f||_\infty+f(Y^2_{t-})}.$$
According to Theorem~II.7.4 of \cite{ikeda_stochastic_1989}, there exists a Poisson measure $\pi$ on $\r_+\times\r_+$ having intensity $dt\cdot dz$ such that, for all $E\in\mathcal{B}(\r^2),$
\begin{equation}\label{eq:firstrepr}
\mu^Y([0,t]\times E)=\int_0^t\int_0^\infty\uno{\theta(s,z)\in E}\pi(ds,dz).
\end{equation}

Now let us consider two independent Poisson measures $\tilde\pi^1,\tilde\pi^2$ (independent of everything else) on $[||f||_\infty,\infty[$ having Lebesgue intensity. We then define $\pi^1$ in the following way:  for any $A\in\mathcal{B}(\r_+\times[0,||f||_\infty]),$ $\pi^1(A)=\pi(A)$, and for $A\in\mathcal{B}(\r_+\times]||f||_\infty,\infty[),$ $\pi^1(A)=\tilde\pi^1(A).$ We define $\pi^2$ in a similar way:  For $A\in\mathcal{B}(\r_+\times[0,||f||_\infty]),$ $\pi^2(A)=\pi(\{(t,||f||_\infty+z):(t,z)\in A\})$, and for $A\in\mathcal{B}(\r_+\times]||f||_\infty,\infty[),$ $\pi^2(A)=\tilde\pi^2(A).$ By definition of Poisson measures, $\pi^1$ and $\pi^2$ are independent Poisson measures on $\r_+^2$ having Lebesgue intensity, and together with \eqref{eq:firstrepr},  we have
$$M^{d,i}_t=-\int_{[0,t]\times\r_+}Y^i_{s-}\uno{z\leq f(Y^i_{s-})}\pi^i(ds,dz)+\int_0^tY_s^if(Y^i_s)ds,~~1\leq i\leq 2.$$
\end{proof}

Moreover we have the following
\begin{thm}
\label{convergencemartingale}
Assume that Assumptions \ref{ass:1}, \ref{control} and \ref{ass:2} hold. Then the distribution of any limit $\mu$ of the sequence $ \mu^N :=\frac1{N}\sum_{j=1}^{N}\delta_{X^{N,j}} $ is solution of item (ii) of $({\mathcal M}).$
\end{thm}

\begin{proof}
{\it Step~1.} We first check that for any $t\geq 0$, a.s., $\mu(\{\gamma \, : \, \Delta\gamma(t)\neq 0\})=0$.
We assume by contradiction that there exists $t > 0 $ 
such that $\mu ( \{ \gamma  : \Delta \gamma (t) \neq 0 \} ) > 0 $
with positive probability. Hence there are $a,b>0$ such that the event 
$E:=\{\mu ( \{ \gamma : |\Delta \gamma (t) |  > a  \} ) > b\}$ has a 
positive probability. For every $\varepsilon > 0$, we have
$E\subset \{ \mu (  \cB^\varepsilon_a  ) > b\}$, where 
$\cB^\varepsilon_a := \{ \gamma : \sup_{ s \in (t- \varepsilon , t + \varepsilon)}| \Delta \gamma(s) | > a \}$, which
is an open subset of $D(\R_+,\R)$. Thus $\cP^{\varepsilon}_{a,b} := 
\{  \mu \in {\cP} ( {D} ( \R_+, \R ) ) : \mu (  \cB^\varepsilon_a  ) > b \}$
is an open subset of $ {\cP} ( {\D} ( \R_+,\r) )$. 
The Portmanteau theorem implies then that for any $\varepsilon>0$,
\begin{equation}
\label{controlPE}
\liminf_{N \to \infty } P ( \mu^N \in \cP^{\varepsilon}_{a,b}  ) \geq P ( \mu \in \cP^{\varepsilon}_{a,b}  ) 
\geq P ( E)  > 0.
\end{equation}
Firstly, we can write
$$J^{N,\eps,i}:=\underset{t-\eps<s<t+\eps}{\sup}\ll|\Delta X^{N,i}_s\rr|=G^{\eps,i}_N\vee S^\eps_N,$$
where $G^{\eps,i}_N:=\max_{s\in D^{\eps,i}_N}|X^{N,i}_{s-}|$ is the maximal height of the big jumps of $X^{N,i},$ with $D^{\eps,i}_N:=\{t-\eps\leq s\leq t+\eps:\pi^i(\{s\}\times[0,f(X^{N,i}_{s-})]\times\r_+)\neq 0\}.$ Moreover,  $S^{\eps}_N:=\max\{|U^j(s)|/\sqrt{N}:s\in\bigcup_{1\leq j\leq N} D^{\eps,j}_N\}$ is the maximal height of the small jumps of $X^{N,i}$, where $U^j(s)$ is defined for $s\in D^{\eps,j}_N$, almost surely, as the only real number that satisfies $\pi^j(\{s\}\times[0,f(X^{N,j}_{s-})]\times\{U^j(s)\})=1.$

We have that
$$\ll\{\mu^N(\mathcal{B}^\eps_a)>b\rr\}= \ll\{\frac1N\sum_{j=1}^N\uno{J^{N,\eps,j}>a}> b  \rr\}.$$
Consequently, by exchangeability and Markov's inequality,
\begin{equation}
\label{probamuN}
\pro{\mu^N(\mathcal{B}^\eps_a)>b}\leq \frac1b\esp{\uno{J^{N,\eps,1}>a}}= \frac1b\pro{J^{N,\eps,1}>a}\leq\frac1b\ll(\pro{G_N^{\eps,1}>a}+\pro{S_N^\eps>a}\rr).
\end{equation}

The number of big jumps of $X^{N,1}$ in $]t-\eps,t+\eps[$ is smaller than a random variable $\xi$ having Poisson distribution with parameter $2\eps||f||_\infty.$ Hence
\begin{equation}
\label{probaGN}
\pro{G_{N,1}^\eps>a}\leq \pro{\xi\geq 1}=1-e^{2\eps||f||_\infty}\leq 2\eps||f||_\infty.
\end{equation}

The small jumps that occur in $]t-\eps,t+\eps[$ are included in $\{U_1/\sqrt{N},...,U_K/\sqrt{N}\}$ where $K$ is a $\n-$valued random variable having Poisson distribution with parameter $2\eps N||f||_\infty$, which is independent of the variables $U_i$ ($i\geq 1$) that are i.i.d. with distribution~$\nu.$ Hence, 
\begin{equation*}
\pro{S_N^\eps>a}\leq \pro{\underset{1\leq i\leq K}{\max}\frac{|U_i|}{\sqrt{N}}>a}\leq \esp{\proc{\underset{1\leq i\leq K}{\max}\frac{|U_i|}{\sqrt{N}}>a}{K}}=\esp{\psi(K)},
\end{equation*}
where $\psi(k)=\pro{\max_{1\leq i\leq k}|U_i|>a\sqrt{N}}\leq k\pro{|U_1|>a\sqrt{N}}\leq ka^{-2}N^{-1}\esp{U_1^2}.$ Hence
\begin{equation}
\label{probaSN}
\pro{S_N^\eps>a}\leq \frac{\esp{U_1^2}}{Na^2}\esp{K}\leq 2||f||_\infty\esp{U_1^2}\frac{1}{a}\eps.
\end{equation}

Inserting the bounds~\eqref{probaGN} and~\eqref{probaSN} in~\eqref{probamuN}, we have
$$\pro{\mu^N(\mathcal{B}^\eps_a>b)}\leq C\eps,$$
where $C$ does not depend on $N$ nor $\eps.$ This last inequality is in contradiction with~\eqref{controlPE} since $P(E)$ does not depend on~$\eps.$

\noindent{\it Step 2.} In the following, we note $\partial^2\phi:=\sum_{i,j=1}^2\partial^2_{x^i x^j}\phi.$ For any $ 0 \le s_1 < \ldots < s_k < s < t$, any $\varphi_1,\dots,\varphi_k 
,\psi_1,\hdots,\psi_k\in C_b ( \R)$, any $\varphi \in C^3_b (\R^2)$, we introduce
\begin{multline*}
F(\mu):=\psi_1(\mu_{s_1}(f)) \hdots\psi_k(\mu_{s_k}(f) )\int_{D(\r_+,\r)^2} \mu \otimes \mu (d\gamma)\phi_1(\gamma_{s_1})\hdots\phi_k(\gamma_{s_k})\\
\ll[\phi(\gamma_t)-\phi(\gamma_s)+\alpha\int_s^t\gamma^1_r\partial_{x^1} \varphi (\gamma_r)dr+\alpha\int_s^t\gamma^2_r\partial_{x^2}\phi(\gamma_r)dr-\frac{\sigma^2}{2}\int_s^t \mu_r(f)\partial^2\phi(\gamma_r)dr\rr.\\
\ll.-\int_s^tf(\gamma^1_r)(\phi(0,\gamma^2_r)-\phi(\gamma_r))dr-\int_s^tf(\gamma^2_r)(\phi(\gamma^1_r,0)-\phi(\gamma_r))dr\rr].
\end{multline*}

To show that $\mathcal{L}(\mu)$ is solution of item (ii) of the martingale problem~$(\mathcal{M}),$ by a classical density argument, it is sufficient  to prove that $\esp{F(\mu)}=0.$ We have
\begin{multline*}
F(\mu^N)=\psi_1(\mu^{N}_{s_1}(f))\hdots\psi_k(\mu^{N}_{s_k}(f)) 
\frac{1}{N^2}\sum_{i =1}^{N} \sum_{j=1}^N \phi_1(X^{N,i}_{s_1}, X^{N, j }_{s_1} )\hdots\phi_k(X^{N,i}_{s_k},X^{N,j}_{s_k}) \cdot\\
\ll[\phi(X^{N,i}_t,X^{N,j}_t)-\phi(X^{N,i}_s,X^{N,j}_s)+\alpha\int_s^tX^{N,i}_r\partial_{x^1}\phi(X^{N,i}_r,X^{N,j}_r)dr\rr.\\
+\alpha\int_s^tX^{N,j}_r\partial_{x^2}\phi(X^{N,i}_r,X^{N,j}_r)dr-\frac{\sigma^2}{2}\int_s^t \mu^N_r (f)  \partial^2\phi(X^{N,i}_r,X^{N,j}_r)dr\\
\ll. -\int_s^tf(X^{N,i}_r)(\phi(0,X^{N,j}_r)-\phi(X^{N,i}_r,X^{N,j}_r))dr
-\int_s^tf(X^{N,j}_r)(\phi(X^{N,i}_r,0)-\phi(X^{N,i}_r,X^{N,j}_r))dr\rr].
\end{multline*}
But recalling~\eqref{eq:dyn} and using Ito's formula, for any $ i \neq j,$ we have
\begin{multline*}
\phi(X^{N,i}_t,X^{N,j}_t)\\
=\phi(X^{N,i}_s,X^{N,j}_s)
-\alpha\int_s^tX^{N,i}_r\partial_{x^1}\phi(X^{N,i}_r,X^{N,j}_r)dr-\alpha\int_s^tX^{N,j}_r\partial_{x^2}\phi(X^{N,i}_r,X^{N,j}_r)dr\\
+\int_{]s,t]\times\r_+\times\r}\uno{z\leq f(X^{N,i}_{r-})}\ll[\phi\ll(0,X^{N,j}_{r-}+\frac{u}{\sqrt{N}}\rr)-\phi(X^{N,i}_{r-},X^{N,j}_{r-})\rr]\pi^{i}(dr,dz,du)\\
+\int_{]s,t]\times\r_+\times\r}\uno{z\leq f(X^{N,j}_{r-})}\ll[\phi\ll(X^{N,i}_{r-}+\frac{u}{\sqrt{N}},0\rr)-\phi(X^{N,i}_{r-},X^{N,j}_{r-})\rr]\pi^{j}(dr,dz,du)\\
+\sum_{\substack{k=1\\k\not\in\{i,j\}}}^{N}\int_{]s,t]\times\r_+\times\r}\uno{z\leq f(X^{N,k}_{r-})}\ll[\phi\ll(X^{N,i}_{r-}+\frac{u}{\sqrt{N}},X^{N,j}_{r-}+\frac{u}{\sqrt{N}}\rr)\rr.\\
\ll.-\phi(X^{N,i}_{r-},X^{N,j}_{r-})\rr]\pi^{k}(dr,dz,du).
\end{multline*}
We use the notation $\tilde\pi^j(dr,dz,du)=\pi^j(dr,dz,du)-drdz\nu(du)$ and set
\begin{align*}
M^{N,i,j,1}_{s,t}:=&\int_{]s,t]\times\r_+\times\r}\uno{z\leq f(X^{N,i}_{r-})}\ll[\phi\ll(0,X^{N,j}_{r-}+\frac{u}{\sqrt{N}}\rr)-\phi(X^{N,i}_{r-},X^{N,j}_{r-})\rr]\tilde\pi^{i}(dr,dz,du),\\
M^{N,i,j,2}_{s,t}:=&\int_{]s,t]\times\r_+\times\r}\uno{z\leq f(X^{N,j}_{r-})}\ll[\phi\ll(X^{N,i}_{r-}+\frac{u}{\sqrt{N}},0\rr)-\phi(X^{N,i}_{r-},X^{N,j}_{r-})\rr]\tilde\pi^{j}(dr,dz,du),\\
W^{N,i,j}_{s,t}:=&\sum_{\substack{k=1\\j\not\in\{i,j\}}}^{N}\int_{]s,t]\times\r_+\times\r}\uno{z\leq f(X^{N,k}_{r-})}\ll[\phi\ll(X^{N,i}_{r-}+\frac{u}{\sqrt{N}},X^{N,j}_{r-}+\frac{u}{\sqrt{N}}\rr)\rr.\\
&\hspace*{8cm}\ll.-\phi(X^{N,i}_{r-},X^{N,j}_{r-})\rr]\tilde\pi^{k}(dr,dz,du),\\
\Delta^{N,i,j,1}_{s,t}:=&\int_s^t\int_\r f(X^{N,i}_r)\ll[\phi\ll(0,X^{N,j}_r+\frac{u}{\sqrt{N}}\rr)-\phi(0,X^{N,j}_r)\rr]d\nu(u)dr,\\
\Delta^{N,i,j,2}_{s,t}:=&\int_s^t\int_\r f(X^{N,j}_r)\ll[\phi\ll(X^{N,i}_r+\frac{u}{\sqrt{N}},0\rr)-\phi(X^{N,i}_r,0)\rr]d\nu(u)dr,\\
\Gamma^{N,i,j}_{s,t}:=&\sum_{\substack{k=1\\k\not\in\{i,j\}}}^{N}\int_s^t\int_\r f(X^{N,k}_r)\ll[\phi\ll(X^{N,i}_r+\frac{u}{\sqrt{N}},X^{N,j}_r+\frac{u}{\sqrt{N}}\rr)-\phi(X^{N,i}_r,X^{N,j}_r)\rr.\\
&\ll.-\frac{u}{\sqrt{N}}\partial_{x^1}\phi(X^{N,i}_r,X^{N,j}_r)-\frac{u}{\sqrt{N}}\partial_{x^2}\phi(X^{N,i}_r,X^{N,j}_r)\rr]d\nu(u)dr \\
&-\int_s^t\int_\r\frac{u^2}{2}\partial^2\phi(X^{N,i}_r,X^{N,j}_r)\frac1N\sum_{\substack{k=1\\k\not\in\{i,j\}}}^{N}f(X^{N,k}_r)d\nu(u)dr,\\
R^{N,i,j}_{s,t}:=&\frac{\sigma^2}{2}\int_s^t\partial^2\phi(X^{N,i}_r,X^{N,j}_r)\ll(\frac1N\sum_{\substack{k=1\\k\not\in\{i,j\}}}^{N}f(X^{N,k}_r)-\frac{1}{N}\sum_{k=1}^{N}f(X^{N,k}_r)\rr)dr.
\end{align*}
Finally, for $ i = j , $ we have
\begin{multline*}
\phi(X^{N,i}_t,X^{N,i}_t)=\phi(X^{N,i}_s,X^{N,i}_s)\\
-\alpha\int_s^tX^{N,i}_r\partial_{x^1}\phi(X^{N,i}_r,X^{N,i}_r)dr-\alpha\int_s^tX^{N,i}_r\partial_{x^2}\phi(X^{N,i}_r,X^{N,i}_r)dr\\
+\int_{]s,t]\times\r_+\times\r}\uno{z\leq f(X^{N, i}_{r-})}\ll[\phi\ll(0,0 \rr)-\phi(X^{N,i}_{r-},X^{N,i}_{r-})\rr]\pi^{i}(dr,dz,du)\\
+\sum_{\substack{k=1\\k\neq i}}^{N}\int_{]s,t]\times\r_+\times\r}\uno{z\leq f(X^{N,k}_{r-})}\ll[\phi\ll(X^{N,i}_{r-}+\frac{u}{\sqrt{N}},X^{N,i}_{r-}+\frac{u}{\sqrt{N}}\rr)
-\phi(X^{N,i}_{r-},X^{N,i}_{r-})\rr]\pi^{k}(dr,dz,du).
\end{multline*}
The associated martingales and error terms are given by 
\begin{align*}
M^{N, i }_{s,t}:= & \int_{]s,t]\times\r_+\times\r}\uno{z\leq f(X^{N, i}_{r-})}\ll[\phi\ll(0,0 \rr)-\phi(X^{N,i}_{r-},X^{N,i}_{r-})\rr]\tilde \pi^{i}(dr,dz,du),\\
W^{N, i }_{s,t} := & \sum_{\substack{k=1\\k\neq i}}^{N}\int_{]s,t]\times\r_+\times\r}\uno{z\leq f(X^{N,k}_{r-})}\ll[\phi\ll(X^{N,i}_{r-}+\frac{u}{\sqrt{N}},X^{N,i}_{r-}+\frac{u}{\sqrt{N}}\rr) \rr.\\
&\hspace*{8cm}  \ll. -\phi(X^{N,i}_{r-},X^{N,i}_{r-})\rr]\tilde \pi^{k}(dr,dz,du),\\
\Delta^{N,i}_{s,t}:=&\int_s^t\int_\r f(X^{N,i}_r)\ll[\phi\ll(0,0 \rr)-\phi(0,X^{N,i}_r) - \phi ( X^{N, i }_r, 0) + \phi (X^{N, i }_r, X^{N, i }_r)  \rr]d\nu(u)dr,\\
\Gamma^{N,i}_{s,t}:=&\sum_{\substack{k=1\\k\neq i }}^{N}\int_s^t\int_\r f(X^{N,k}_r)\ll[\phi\ll(X^{N,i}_r+\frac{u}{\sqrt{N}},X^{N,i}_r+\frac{u}{\sqrt{N}}\rr)-\phi(X^{N,i}_r,X^{N,i}_r)\rr.\\
&\ll.-\frac{u}{\sqrt{N}}\partial_{x^1}\phi(X^{N,i}_r,X^{N,i}_r)-\frac{u}{\sqrt{N}}\partial_{x^2}\phi(X^{N,i}_r,X^{N,i}_r)\rr]d\nu(u)dr \\
&-\int_s^t\int_\r\frac{u^2}{2}\partial^2\phi(X^{N,i}_r,X^{N,i}_r)\frac1N\sum_{\substack{k=1\\k\neq i }}^{N}f(X^{N,k}_r)d\nu(u)dr,\\
R^{N,i}_{s,t}:=&\frac{\sigma^2}{2}\int_s^t\partial^2\phi(X^{N,i}_r,X^{N,i}_r)\ll(\frac1N\sum_{\substack{k=1\\k\neq i }}^{N}f(X^{N,k}_r)-\frac{1}{N}\sum_{k=1}^{N}f(X^{N,k}_r)\rr)dr.
\end{align*}
Then we obtain, since $\int_\r ud\nu(u)=0$,  that
\begin{multline*}
F(\mu^N)=\psi_1(\mu^{N}_{s_1}(f))\hdots\psi_k(\mu^{N}_{s_k}(f)) 
\frac{1}{N^2}\sum_{i, j =1, i \neq j }^{N}\phi_1(X^{N,i}_{s_1},X^{N,j}_{s_1})\hdots\phi_k(X^{N,i}_{s_k},X^{N,j}_{s_k})\\
\ll[M^{N,i,j,1}_{s,t}+M^{N,i,j,2}_{s,t} +W^{N,i,j}_{s,t} +\Delta^{N,i,j,1}_{s,t}+\Delta^{N,i,j,2}_{s,t}+\Gamma^{N,i,j}_{s,t}+R^{N,i,j}_{s,t}\rr]\\
+\psi_1(\mu^{N}_{s_1}(f))\hdots\psi_k(\mu^{N}_{s_k}(f)) 
\frac{1}{N^2}\sum_{i=1 }^{N}\phi_1(X^{N,i}_{s_1},X^{N,i}_{s_1})\hdots\phi_k(X^{N,i}_{s_k},X^{N,i}_{s_k})\\
\ll[M^{N,i}_{s,t}+W^{N,i}_{s,t}
+\Delta^{N,i}_{s,t}+\Gamma^{N,i}_{s,t}+R^{N,i}_{s,t}\rr].
\end{multline*}

Using exchangeability and the boundedness of the $\phi_j,\psi_j$ ($1\leq j\leq k$) and the fact that $M^{N,i,j,1}, \ldots , $  $W^{N,i}$ are martingales, this implies
\[
|\esp{F(\mu^N)}|\leq C\esp{|\Delta^{N,i,j,1}_{s,t}|+|\Delta^{N,i,j,2}_{s,t}|+|\Gamma^{N,i,j}_{s,t}|+|R^{N,i,j}_{s,t}| +\frac{ |\Delta^{N,i}_{s,t}|+|\Gamma^{N,i}_{s,t}|+|R^{N,i}_{s,t}|}{N}}.
\]
Since $f$ is bounded and $\phi\in C^3_b(\r^2),$ Taylor-Lagrange's inequality implies then that
$$|\esp{F(\mu^N)}|\leq \frac{C}{\sqrt{N}}.$$
Finally, using that $F$ is bounded and almost surely continuous at $\mu$ (see {\it Step~1}), we have
$$\esp{F(\mu)}=\underset{N\rightarrow\infty}{\lim}\esp{F(\mu^N)}=0,$$
concluding our proof.
\end{proof}

Now we have all elements to give the proof of the following main result. 
\begin{thm}\label{uniquelimit}
Grant Assumptions \ref{ass:1}, \ref{control} and \ref{ass:2}. Each converging subsequence of $\mu^N:=\frac1N\sum_{j=1}^N\delta_{X^{N,j}}$ converges to the same limit $ \mu = {\mathcal L} ( \bar X | {\mathcal W}) , $ where $\bar X $ is the unique strong solution of \eqref{eq:dynlimit}.
\end{thm}

\begin{proof}
Let us consider the limit (in distribution) $\mu$ of a subsequence of $\mu^N$.  By Proposition~(7.20) of \cite{aldous_exchangeability_1983}, $\mu$ is the directing measure of some exchangeable system $(\bar Y^i)_{i\geq 1},$ and it holds that, for the chosen subsequence, 
$(X^{N,i})_{1\leq i\leq N}$ converges in law to $(\bar Y^i)_{i\geq 1}.$ Moreover, we also know that 
$$\mu=\mathcal{L}(\bar Y^i| \mu) \mbox{ and  } \mu \otimes \mu = \mathcal{L}((\bar Y^i, \bar Y^j) | \mu) , $$
almost surely, for all $ i \neq j . $ In particular, for all $ i \neq j,$ 
$$ \mathcal {L} ( \mu, (\bar Y^i , \bar Y^j ) ) = P ,$$ 
where $ P$ is given by \eqref{eq:P}, with $ Q = \mathcal{L} ( \mu ) .$ 

Thanks to Lemma~\ref{representation}, together with Theorem~\ref{convergencemartingale}, we know that there exist Brownian motions $W^{(i,j)}$ ($i,j \geq 1$) and Poisson random measures $\pi^i$ ($i\geq 1$) such that for all pairs  $(i,j) , i \neq j ,$  $\pi^{i}$ is independent of $\pi^{j}$ and such that
\begin{align*}
d\bar Y^{i}_t=&-\alpha\bar Y^{i}_tdt+\sigma\sqrt{\mu_t(f)}dW^{(i,j)}_t-\bar Y^{i}_{t-}\int_{\r_+}\uno{z\leq f(\bar Y^{i}_{t-})}\pi^{i}(dt,dz),\\
d\bar Y^{j}_t=&-\alpha\bar Y^{j}_tdt+\sigma\sqrt{\mu_t(f)}dW^{(i,j)}_t-\bar Y^{j}_{t-}\int_{\r_+}\uno{z\leq f(\bar Y^{j}_{t-})}\pi^{j}(dt,dz).
\end{align*}
The exchangeability of the system $(\bar Y^i)_{i\geq 1}$ implies the independence of the $(\pi^i)_{i\geq 1}$ and that for all $i,j,k \geq 1,$ $ i \neq j,$ $i \neq k, $  $W^{(i,j) }=W^{(i, k )}=W.$

The last point to prove is that $\mu_t(f):=\espc{f(\bar Y^1_t)}{\mu}=\espc{f(\bar Y^1_t)}{\W}~a.s..$ This would be a consequence of the fact that, conditionally on $W,$ the processes $(\bar Y^j)_{j\geq 1}$ are $i.i.d.$ (see Lemma~(2.12).(a) of \cite{aldous_exchangeability_1983}). But this last assertion is not trivial because we do not know yet that $W$ is the only noise  term common to each process $\bar Y^j$ ($j\geq 1$). That is why we will introduce an auxiliary particle system which is a mean field version of the limit system and which converges to $(\bar Y^j)_{j\geq 1}.$

To begin with, Lemma~(2.15) of \cite{aldous_exchangeability_1983} implies that $\mu_t(f)$ is the almost sure limit of $N^{-1}\sum_{j=1}^Nf(\bar Y^j_t).$ Now, let us prove that this sequence converges to $\espc{f(\bar Y^1_t)}{\W}.$ For this purpose, we introduce the system $(\auxY^{N,i})_{1\leq i\leq N},$ driven by the same Brownian motion $ W$ and the same Poisson random measures $\pi^i,$ with $\bar Y^i_0=\tilde X^{N,i}_0$ ($i\geq 1$), replacing the term $\mu_t(f)$ by the empirical measure:
$$d\auxY_t^{N,i}=-\alpha\auxY^{N,i}_tdt+\sqrt{\frac1N\sum_{j=1}^Nf(\auxY^{N,j}_t)}dW_t-\auxY^{N,i}_{t-}\int_{\r_+}\uno{z\leq f(\auxY^{N,i}_{t-})}\pi^i(dt,dz), \; \auxY_0^{N,i} = \bar Y_0^i .$$
Consider finally the system $(\bar X^i)_{i\geq 1}$ defined in \eqref{eq:dynlimit1}, driven by the same Brownian motion $ W$ and the same Poisson random measures $\pi^i $ as $(\bar Y^i)_{i\geq 1},$ with $ \bar X^i_0 = \bar Y^i_0 $ for all $ i \geq 1.$  In this way, $ (\bar X^i)_{i \geq 1} , (\bar Y^i)_{i \geq 1} $ and $ (\tilde X^{N, i })_{1 \le i \le N} $ are all defined on the same probability space.

It is now sufficient to prove that both for $(\bar Y^i)_{i\geq 1}$ and for  $(\bar X^i)_{i\geq 1},$
\begin{equation}
\label{auxlim}
\esp{\ll|a(\bar Y^i_t)-a(\auxY^{N,i}_t)\rr|} + \esp{\ll|a(\bar X^i_t)-a(\auxY^{N,i}_t)\rr|}\leq C_t N^{-1/2}.
\end{equation}
Indeed, suppose we have already proven the above control \eqref{auxlim}. Then 
\begin{multline*}
\esp{\ll|\frac1N\sum_{j=1}^Nf(\bar Y^j_t)-\espc{f(\bar X^1_t)}{\W}\rr|}\leq \frac1N\sum_{j=1}^N\esp{|f(\bar Y^j_t)-f(\auxY^{N,j}_t)|}\\
+\frac1N\sum_{j=1}^N\esp{|f(\auxY^{N,j}_t) - f( \bar X_t^j ) |}+\esp{\ll|\frac1N\sum_{j=1}^Nf(\bar X^{j}_t)-\espc{f(\bar X^{1}_t)}{\W}\rr|}
.
\end{multline*}
Then, \eqref{auxlim} and Assumption~\ref{ass:2} imply that the first term and the second one of the sum above are smaller than $C_tN^{-1/2}$ for some $C_t>0.$ In addition, by item (ii) of Proposition \ref{independence}, the variables $  (\bar X^j)_{ 1 \le j \le N } $ are i.i.d., conditionally on $W.$ Consequently, the third term is also smaller than $C_tN^{-1/2}.$

This implies that 
$$\mu_t ( f) = \espc{f(\bar Y^1_t)}{\mu}=\espc{f(\bar X^1_t)}{\W}= \espc{f(\bar X^i_t)}{\W} \mbox{ a.s..}$$
As a consequence, $(\bar Y^i)_{i\geq 1}$ is solution of the infinite system 
$$ d\bar Y^{i}_t=-\alpha\bar Y^{i}_tdt+\sigma\sqrt{\espc{f(\bar X^i_t)}{\W}}dW_t-\bar Y^{i}_{t-}\int_{\r_+}\uno{z\leq f(\bar Y^{i}_{t-})}\pi^{i}(dt,dz),$$
while $(\bar X^i)_{i\geq 1}$ in \eqref{eq:dynlimit1} is solution of 
$$ d\bar X^{i}_t=-\alpha\bar X^{i}_tdt+\sigma\sqrt{\espc{f(\bar X^i_t)}{\W}}dW_t-\bar X^{i}_{t-}\int_{\r_+}\uno{z\leq f(\bar X^{i}_{t-})}\pi^{i}(dt,dz),$$
with $ \bar X^i_0 = \bar Y^i _0,$ for all $ i \geq 1.$

Let us prove that $\bar X^i=\bar Y^i$ almost surely. For that sake, consider $\tau_M=\inf\{t>0:|\bar X^i_t|\wedge|\bar Y^i_t|>M\}$ for $M>0.$ We prove that $\esp{|\bar X^i_{t\wedge\tau_M}-\bar Y^i_{t\wedge\tau_M}|}=0$ for all $M>0,$ which implies, by Fatou's lemma, that $\esp{|\bar X^i_{t}-\bar Y^i_{t}|}=0$, recalling~\eqref{controllimite}, and the fact that we can prove a similar control for $\bar Y^i$. Let $u_M(t):=\esp{|\bar X^i_{t\wedge\tau_M}-\bar Y^i_{t\wedge\tau_M}|}$. To see that $u_M(t)=0$, it is sufficient to apply Gr\"onwall's lemma to the following inequality
$$
u_M(t)\leq \alpha\int_0^tu_M(s)ds+\esp{\int_{[0,t\wedge\tau_M]\times\r_+}\ll|\bar X^i_{s-}\uno{z\leq f(\bar X^i_{s-})}-\bar Y^i_{s-}\uno{z\leq f(\bar Y^i_{s-})}\rr|\pi^i(ds,dz)}
$$ 
implying that 
\begin{multline*}
u_M(t) \leq \alpha\int_0^tu_M(s)ds+\esp{\int_{[0,t\wedge\tau_M]\times \R_+ } \indiq_{ z \in [0,f(\bar X^i_{s-})\wedge f(\bar Y^i_{s-})]} \ll|\bar X^i_{s-}-\bar Y^i_{s-}\rr|\pi^i(ds,dz)}\\
+\esp{\int_{[0,t\wedge\tau_M]\times \R_+} \indiq_{ z \in ]f(\bar X^i_{s-})\wedge f(\bar Y^i_{s-}), f(\bar X^i_{s-})\vee f(\bar Y^i_{s-})]}  |\bar X^i_{s-}|\vee|\bar Y^i_{s-}|\pi^i(ds,dz)},
\end{multline*}
whence 
$$ 
u_M(t)\leq C(1+M)\int_0^tu_M(s)ds.
$$

Hence $(\bar Y^i)_{i\geq 1}$ is solution of the infinite system \eqref{eq:dynlimit1} and $ \mu =\mathcal{L}(\bar Y^1|\W),$ its directing measure, is uniquely determined. As a consequence,  $\mu^N$ converges in distribution to $\mathcal{L}(\bar Y^1|\W)$ in $\mathcal{P}(D(\r_+,\r)).$

Let us now show \eqref{auxlim}. We only prove it for $ \bar Y^i, $ the proof for $ \bar X^i$ is similar. We decompose the evolution of $a(\bar Y^{1}_t)$ in the following way. 
\begin{align}
a(\bar Y^{1}_t) &=  a(\bar Y^{1}_0)  -\alpha  \int_0^t  a'(\bar Y^{1}_s)\bar Y^{1}_s  ds+ \int_{[0,t]\times\r_+} \ll(a(0)-a(\bar Y^{1}_{s-})\rr)  \indiq_{ \{ z \le  f ( \bar Y^{1}_{s-}) \}} \bN^1 (ds,dz ) \label{itobarY}\\
&+\frac{\sigma^2}{2}\int_0^ta''(\bar Y^{1}_s)\frac{1}{N}\sum_{j=1}^Nf(\bar Y^{j}_s)ds - B_t^N +\sigma \int_0^t a'(\bar Y^{1}_s)\sqrt{ \frac1N \sum_{j=1}^N f ( \bar Y_s^{j} )  } d W_{s} -M_t^N, 
\nonumber
\end{align}
where
$$B_t^N=\frac{\sigma^2}{2}\int_0^ta''(\bar Y^{1}_s)\ll(\frac{1}{N}\sum_{j=1}^Nf(\bar Y^{j}_s)-\espc{f(\bar Y^{1}_s)}{\mu}\rr)ds$$
and
$$ M_t^N = \sigma \int_0^t a'(\bar Y^{1}_s)\left( \sqrt{ \frac1N \sum_{j=1}^N  f ( \bar Y_s^{j} )  } -  \sqrt{\esp{ f ( \bar Y_s^{1} )  | \mu} }\right)  d W_s.$$
Since
$$ <M^N>_t \leq \sigma^2 \ll(\underset{x\in\r}{\sup}|a'(x)^2|\rr) \int_0^t \left( \sqrt{ \frac1N \sum_{j=1}^N  f ( \bar Y_s^{j} )  } -  \sqrt{\esp{  f ( \bar Y_s^{1} )  | \mu} }\right)^2 ds, $$
recalling that the variables $\bar Y_s^{j}$ ($1\leq j\leq N$) are i.i.d. conditionally to $\mu$, taking conditional expectation $\E ( \cdot | \mu ) $ implies that 
$$ \esp{<M^N>_t} \le C_t N^{-1}\textrm{ and }\esp{B_t^N} \le C_t N^{-1}.$$
Then, applying Ito's formula on $\tilde X^{N,1}$, we obtain the same equation as~\eqref{itobarY}, but without the terms $B^N_t$ and $M^N_t.$ Introducing
$$u(t):=\underset{0\leq s\leq t}{\sup}\esp{\ll|a(\bar Y^1_s)-a(\tilde X^{N,1}_s)\rr|},$$
we can prove with the same reasoning as in the proof of Theorem~\ref{prop:42} that
$$u(t)\leq C(1+t)u(t)+\frac{C_t}{\sqrt{N}},$$
where $C$ and $C_t$ are independent of~$N$. Finally, using the arguments of the proof of Theorem~\ref{prop:42}, this implies~\eqref{auxlim}.
\end{proof}

Let us end this section with the
\begin{proof}[Proof of Theorem~\ref{convergencemuN}]
According to Proposition~\ref{mutight}, the sequence $(\mu^N)_N$ is tight. Besides, thanks to Theorem~\ref{uniquelimit}, every converging subsequence of $(\mu^N)_N$ converges to the same limit $\mu=\mathcal{L}(\bar X^1|\W)$, where $(\bar X^j)_{j\geq 1}$ is solution of \eqref{eq:dynlimit1}. This implies the result.
\end{proof}

\section{Appendix}

\subsection{Well-posedness of the limit equation~\eqref{eq:dynlimit}}\label{sec:wellposed}

\begin{proof}[Proof of Item 2. of Theorem \ref{prop:42}]
The proof is done using a classical Picard-iteration. For that sake we introduce the sequence of processes $ \limY_t^{[0] } \equiv \limY_0 , $ and 
$$  \limY^{[n+1]}_t := \limY_0 -\alpha \int_0^t \limY_s^{[n]} ds - \int_{[0,t]\times\r_+\times\r} \limY^{[n+1]}_{s- } \indiq_{ \{ z \le  f ( \limY^{[n]}_{s-}) \}} \bN (ds,dz,du) + \sigma \int_0^t \sqrt{ \mu^{n}_s ( f) } d W_s ,$$
where 
$$ \mu_s^n = P ( \limY_s^{[n]} \in \cdot | \W ) .$$ 
Let us first prove a control on the moments of $\limY^{[n]}$ uniformly in $n$. Applying Ito's formula we have
\begin{multline*}
\esp{\ll(\limY_{t}^{[n+1]}\rr)^2}\leq \esp{\limY_0^2}-2\alpha\int_0^t\esp{\ll(\limY_{s}^{[n+1]}\rr)^2}ds+\sigma^2\int_0^t\esp{\mu^n_{s}(f)}ds\\
\leq \esp{\limY_0^2}+\sigma^2\int_0^t\esp{\mu^n_{s}(f)}ds.
\end{multline*}
Using that $f$ is bounded, 
$$\esp{\ll(\limY_{t}^{[n+1]}\rr)^2}\leq \esp{\limY_0^2}+\sigma^2||f||_\infty t,$$
implying
\begin{equation}
\label{controlY[n]}
\underset{n\in\n}{\sup}~\underset{0\leq s\leq t}{\sup}\esp{\ll(\limY_{s}^{[n]}\rr)^2}<+\infty.
\end{equation}
Now, we prove the convergence of $\limY^{[n]}_t$. The same strategy as the one of the proof of Item 1. of Theorem~\ref{prop:42} allows to show that 
$$\delta_t^n := \E \sup_{s \le t } | a ( \limY_s^{[n]} ) - a( \limY_s^{[n-1]} ) | \; \mbox{  satisfies } \; 
  \delta_t^n \le C (t + \sqrt{t} )  \delta_t^{n-1}   ,$$
for all $ n \geq 1 , $ for a constant $C$ only depending on the parameters of the model, but not on $ n, $ neither on $t. $ Choose $t_1 $ such that 
$$  C (t_1  + \sqrt{t_1} )  \le \frac12.$$
Since $ \sup_{s \le t_1 } | a ( \limY_s^{[0]} ) | = a ( \limY_0) \le \| a \|_\infty, $ we deduce from this
that
$$ \delta_{t_1}^n \le 2^{- n } \| a \|_\infty .$$
This implies the almost sure convergence of $a\ll(\limY_t^{[n]}\rr)_n$ to some random variable $Z_t$ for all $t\in[0,t_1]$. As $a$ is an increasing function, the almost sure convergence of $\limY_t^{[n]}$ to some (possibly infinite) random variable $\limY_t$ follows from this. The almost sure finiteness of $\limY_t$ is then guaranteed by Fatou's lemma and~\eqref{controlY[n]}.

Now let us prove that $\limY$ is solution of the limit equation~\eqref{eq:dynlimit} which follows by standard arguments (note that the jump term does not cause troubles because it is of finite activity).  
The most important point is to notice that
$$ \mu_t^n ( f) = \E ( f ( \limY_t^{[n]}) | \W) \to \E ( f (\limY_t) | \W ) $$
almost surely, which follows from the almost sure convergence of $ f ( \limY_t^{[n]} ) \to f (\limY_t ) ,$ using dominated convergence. 

Once the convergence is proven on the time interval $ [0, t_1 ], $ we can proceed iteratively over successive intervals $ [ k t_1, (k+1) t_1] $ to conclude that $\bar X$ is solution of~\eqref{eq:dynlimit} on~$\r_+.$

It remains to prove~\eqref{controllimite}. Firstly, by Fatou's lemma and~\eqref{controlY[n]}, we know that, for all $t>0,$
\begin{equation}
\label{controllimite2}
\underset{0\leq s\leq t}{\sup}\esp{\bar X_s^2}<\infty.
\end{equation}

Besides, Ito's formula gives
\begin{multline*}
\bar X_t^2 = \bar X_0^2 -2\alpha\int_0^t\bar X_s^2ds-\int_{[0,t]\times\r_+\times\r}\bar X_{s-}^2\indiq_{ \{ z \le  f ( \bar X_{s-}) \}} \pi(ds,dz,du)\\
+\sigma^2\int_0^t\mu_s(f)ds+2\sigma\int_0^t\sqrt{\mu_s(f)}\bar X_sdW_s\\
\leq \bar X_0^2+\sigma^2||f||_\infty t+2\sigma\int_0^t\sqrt{\mu_s(f)}\bar X_sdW_s.
\end{multline*}
Inequality~\eqref{controllimite} is then a straightforward consequence of Burkholder-Davis-Gundy inequality,~\eqref{controllimite2} and the above computation.
\end{proof}

We now give the 
\begin{proof}[Proof of Proposition~\ref{independence}]
$(i)$ Given a Brownian motion $W$ and i.i.d. Poisson measures $\pi^i ,$ the same proof as the one of  Theorem~\ref{prop:42} implies the existence and the uniqueness of the system given in~\eqref{eq:dynlimit1} for $1\leq i\leq N.$ 

$(ii)$ The construction of the proof of Item 2. of Theorem~\ref{prop:42}, together with the proof of Theorem~1.1 of Chapter IV.1 and of Theorem 9.1 in Chapter IV.9 of \cite{ikeda_stochastic_1989}, imply the existence of a measurable function $\Phi$ that does not depend on $k=1,\ldots, N$, and that satisfies, for each $1\leq k\leq N,$
$$\limY^k=\Phi(\limY^k_0,\pi^k,W)$$
and for all $ t \geq 0, $ 
\begin{equation}\label{eq:nonanticipatif}
 \limY^k_{|[0,t]} = \Phi_t (\limY^k_0, \pi^k_{| [ 0, t ] \times \R_+ \times \R } ,(W_s)_{ s \le t } ) ;
\end{equation} 
in other words, our process is non-anticipative and does only depend on the underlying noise up to time $t.$ 

Then we can write, for all continuous bounded functions $g,h$,
$$\espc{g(\limY^i)h(\limY^j)}{\W}=\espc{g(\Phi(\limY^i_0,\pi^i,W))h(\Phi(\limY^j_0,\pi^j,W))}{\W}=\psi(W),$$
where $\psi(w):=\esp{g(\Phi(\limY^i_0,\pi^i,w))h(\Phi(\limY^j_0,\pi^j,w))}=\esp{g(\Phi(\limY^i_0,\pi^i,w))}\esp{h(\Phi(\limY^j_0,\pi^j,w))}=:\psi_i(w)\psi_j(w)$. With the same reasoning, we show that $\espc{g(\limY^i)}{\W}=\psi_i(W)$ and $\espc{h(\limY^j)}{\W}=\psi_j(W)$.
The same arguments prove the mutual independence of $\limY^1,\ldots \limY^N$  conditionally to $W.$

$(iii)$ Using the representation $\limY^k_{|[0,t]}=\Phi_{t}(\limY^k_0,\pi^k,W) ,$  we can write for any continuous and bounded function $g : D([0,t],\r) \to \R,$ 
$$\int_\r g d (N^{-1}  \sum_{i=1}^N \delta_{\limY^i_{||0,t]}})  = \frac 1N \sum_{i=1}^N g (\limY^i_{|[0,t]})=\frac 1N \sum_{i=1}^N g \circ\Phi_t(\limY^i_0,\pi^i,W).$$
Using the law of large numbers on the account of the sequence of i.i.d. PRM's and working conditionally on $ W, $ we obtain that 
$$ \lim_{ N \to \infty} \int_\r g d (N^{-1}  \sum_{i=1}^N \delta_{\limY^i_{|[0,t]}}) = \esp{ g \circ\Phi_t(\limY^1_0,\pi^1,W) | \W }  =\esp{ g  ( \limY^1_{|[0,t]})  | \W  } =  \esp{ g ( \limY^1_{|[0,t]})  | (W_s)_{s \le t }  } ,$$
where we have used \eqref{eq:nonanticipatif}.  
\end{proof}

\subsection{Proof of Corollary \ref{cor:PDE}}

Applying Ito's formula, we have
\begin{multline}\label{eq:pde}
 \phi ( \limY_t ) = \phi ( \limY_0 ) + \int_0^t \left( -\alpha\phi' ( \limY_s)\limY_s+ \frac12 \phi'' ( \limY_s)  \mu_s ( f) \right)  ds + 
\int_0^t \phi' ( \limY_s) \sqrt{ \mu_s ( f) } d W_s \\
+ \int_{ [0, t ] \times \R_+ \times \R } \indiq_{\{ z \le f( \limY_{s-} \}} \left( \phi( 0 ) - \phi ( \limY_{s-} \right)  \pi ( ds, dz, du )   .
\end{multline}  
Since $ \phi', \phi''$ and $ f$ are bounded, it follows from~\eqref{controllimite} and Fubini's theorem that
\begin{multline*}
 \E \left( \int_0^t \big( -\alpha\phi' ( \limY_s)\limY_s + \frac12 \phi'' ( \limY_s)  \mu_s ( f) \big)  ds  | W \right) = \int_0^t E \left( -\alpha\phi' ( \limY_s)\limY_s + \frac12 \phi'' ( \limY_s)  \mu_s ( f) | W\right)  ds \\
 = \int_0^t  \int_\R \left( -\alpha\phi' (x)x + \frac12 \phi'' (x)  \mu_s ( f) \right)  \mu_s (dx)  ds .  
 \end{multline*}
Moreover, by independence of $ \limY_0 $ and $ W, $ $ \E ( \phi( \limY_0 ) | W  ) = \int_\R \phi ( x) \nu_0 (dx) .$ 

To deal with the martingale part in \eqref{eq:pde}, we use an Euler scheme to approximate the stochastic integral $I_t := \int_0^t \phi' ( \limY_s) \sqrt{ \mu_s ( f) } d W_s.$ For that sake, let $ t^n_k  : = k 2^{- n } t , 0 \le k \le 2^n,$ $n \geq 1,$ and define 
$$ I_t^n := \sum_{k=0}^{2^n - 1 } \phi ' ( \limY_{t_k^n } ) \Delta_k^n , \;  \Delta_k^n = \int_{t_k^n }^{t_{k+1}^n }    \sqrt{ \mu_s ( f) } d W_s ,$$
then $ \E ( |  I_t - I_t^n |^2 ) \to 0 $ as $ n \to \infty ,$ and therefore $ \E ( I_t^n | \W ) \to \E ( I_t | \W) $ in $L^2 ( P), $ as $ n \to \infty .$ But 
$$  \E ( I_t^n |\W ) = \sum_{k=0}^{2^n - 1 } \E ( \phi ' ( \limY_{t_k^n } ) | \W)   \Delta_k^n \to \int_0^t \E ( \phi' ( \limY_s) | \W)  \sqrt{ \mu_s ( f) } d W_s $$
in $L^2 ( P), $ since the sequence of processes $ Y^n_s :=   \sum_{k=0}^{2^n - 1 } \indiq_{] t_k^n , t_{k+1}^n ] } (s) \E ( \phi ' ( \limY_{t_k^n } ) | \W ) , 0 \le s \le t, $ converges in $L^2 ( \Omega \times [0, t ]) $ to $ \E ( \phi ' ( \limY_s) |\W ) .$ 

We finally deal with the jump part in \eqref{eq:pde}. Since $f$ is bounded, and by independence of $W$ and $\pi, $ we can rewrite this part in terms of an underlying Poisson process $N_t , $ independent of $W $ and having rate $ \| f\|_\infty , $ and in terms of i.i.d.
variables $(V_n)_{n\geq 1}$ uniformly distributed on $  [0, 1 ] ,$  independent of $W$ and of $ N$ as follows. 
$$
 \int_{ [0, t ] \times \R_+ \times \R } \indiq_{\{ z \le f( \limY_{s-} \}} \left( \phi( 0 ) - \phi ( \limY_{s-} \right)  \pi ( ds, dz, du ) =
\sum_{n=1}^{N_t} \indiq_{\{ \|f \|_\infty V_n \le f( \limY_{T_n- } ) \}} ( \phi(0) - \phi(  \limY_{T_n- } )) .
$$ 
Taking conditional expectation $ E ( \cdot | \W), $ we obtain 
\begin{multline*}
\E \left( \sum_{n=1}^{N_t} \indiq_{\{ \|f \|_\infty V_n \le f( \limY_{T_n- } ) \}} ( \phi(0) - \phi(  \limY_{T_n- } ))| \W \right) = \\
\E \left( \sum_{n=1}^{N_t} \frac{ f( \limY_{T_n- } )  }{\|f\|_\infty} ( \phi(0) - \phi(  \limY_{T_n- } ))| \W \right)\\
= \int_0^t \E \left( f( \limY_s) ( \phi(0) - \phi(  \limY_s))| \W \right) ds ,
\end{multline*}
where we have used the independence properties of $ (V_n)_n, N_t $ and $ W $ and the fact that conditionally on $ \{ N_t = n \}, $ the jump times $ (T_1, \ldots , T_n ) $ are distributed as the order statistics of $n$ i.i.d. times which are uniformly distributed on $ [0, t ].$ This concludes our proof.

\subsection{A priori estimates}

In this subsection, we prove useful  a priori upper bounds on some moments of the solutions of the SDEs~\eqref{eq:dyn}.

\begin{lem}
\label{estimate}
Assume that~\ref{control} holds and that $f$ is bounded:
\begin{itemize}
\item[(i)] for all $t>0, $ $\underset{N\in\n^*}{\sup}\underset{0\leq s\leq t}{\sup}\esp{\ll(X_{s}^{N,1}\rr)^2}<+\infty,$  
\item[(ii)] for all $t>0,\underset{N\in\n^*}{\sup}\esp{\underset{0\leq s\leq t}{\sup}\ll|X^{N,1}_s\rr|}<+\infty,$
\end{itemize}
\end{lem}

\begin{proof}
{\it Step~1:} Let us prove $(i)$.
\begin{multline*}
\ll(X_{t}^{N,1}\rr)^2= \ll(X_0^{N,1}\rr)^2-2\alpha\int_0^{t} \ll(X_s^{N,1}\rr)^2ds-\int_{[0,t]\times\r_+\times\r}\ll(X_s^{N,1}\rr)^2\uno{z\leq f\ll(X_{s-}^{N,1}\rr)}d\pi^j(s,z,u)\\
+\sum_{j=2}^N\int_{[0,t]\times\r_+\times\r}\ll[\ll(X_{s-}^{N,1}+\frac{u}{\sqrt{N}}\rr)^2-\ll(X_{s-}^{N,1}\rr)^2\rr]\uno{z\leq f\ll(X_{s-}^{N,j}\rr)}d\pi^j(s,z,u)\\
\leq \ll(X_0^{N,1}\rr)^2+\sum_{j=2}^N\int_{[0,t]\times\r_+\times\r}\ll[\ll(X_{s-}^{N,1}+\frac{u}{\sqrt{N}}\rr)^2-\ll(X_{s-}^{N,1}\rr)^2\rr]\uno{z\leq f\ll(X_{s-}^{N,j}\rr)}d\pi^j(s,z,u).
\end{multline*}
As $f$ is bounded,
$$\esp{\ll(X_{t}^{N,1}\rr)^2}\leq \esp{\ll(X_0^{N,1}\rr)^2}+\frac{\sigma^2}{N}\sum_{j=2}^N\int_0^t\esp{f\ll(X_{s}^{N,j}\rr)}ds\leq \esp{\ll(X_0^{N,1}\rr)^2}+\sigma^2||f||_\infty t.$$
{\it Step~2:} Now we prove $(ii)$.
\begin{align*}
\ll|X_t^{N,1}\rr|\leq& \ll|X_0^{N,1}\rr|+\alpha\int_0^t\ll|X_s^{N,1}\rr|ds+\int_{[0,t]\times\r_+\times\r}\ll|X^{N,1}_{s-}\rr|\uno{z\leq f(X^{N,1}_{s-})}d\pi^1(s,z,u)+\frac{1}{\sqrt{N}}|M_t^N| ,
\end{align*}
where $M_t^N$ is the martingale $M_t^N=\sum_{j=2}^N\int_{[0,t]\times\r_+\times\r}u\uno{z\leq f\ll(X^{N,j}_{s-}\rr)}d\pi^j(s,z,u).$ Then
\begin{multline*}
\underset{0\leq s\leq t}{\sup}\ll|X_s^{N,1}\rr|\leq \ll|X_0^{N,1}\rr|+\alpha\int_0^t|X_s^{N,1}|ds+\int_{[0,t]\times\r_+\times\r}\ll|X^{N,1}_{s-}\rr|\uno{z\leq f(X^{N,1}_{s-})}d\pi^1(s,z,u)\\
+\frac{1}{\sqrt{N}}\underset{0\leq s\leq t}{\sup}|M_s^N|.
\end{multline*}
To conclude the proof, it is now sufficient to notice that
$$\frac{1}{\sqrt{N}}\esp{\underset{0\leq s\leq t}{\sup}|M_s^N|}\leq \esp{\frac{1}{N}[M^N]_t}^{1/2}$$
is  uniformly bounded in $N$, since $f$ is bounded, and to use the point $(i)$ of the lemma.
\end{proof}

\section*{Acknowledgements}

This research has been conducted as part of FAPESP projectResearch, Innovation and Dissemination Center for Neuromathematics(grant 2013/07699-0).

\bibliography{biblio}
\end{document}